\documentclass[fleqn, pdftex]{article} 

\usepackage{amsmath,amssymb,amsthm}
\usepackage{tikz}
\usepackage{pifont}

\title{Topological semantics of conservativity and interpretability logics}
\author{Sohei Iwata\footnote{E-mail: soh.iwata@people.kobe-u.ac.jp}
\footnote{Graduate School of System Informatics, Kobe University, 1-1 Rokkodai, Nada, Kobe 657-8501, Japan.} and Taishi Kurahashi\footnote{Email: kurahashi@people.kobe-u.ac.jp}
\footnote{Graduate School of System Informatics, Kobe University, 1-1 Rokkodai, Nada, Kobe 657-8501, Japan.}}
\date{}

\theoremstyle{plain}
\newtheorem{thm}{Theorem}[section]
\newtheorem{lem}[thm]{Lemma}
\newtheorem{prop}[thm]{Proposition}
\newtheorem{cor}[thm]{Corollary}
\newtheorem{fact}[thm]{Fact}
\newtheorem{prob}[thm]{Problem}
\newtheorem{cl}{Claim}

\theoremstyle{definition}
\newtheorem{defn}[thm]{Definition}

\newcommand{\PA}{\mathbf{PA}}

\newcommand{\GL}{\mathbf{GL}}
\newcommand{\CL}{\mathbf{CL}}
\newcommand{\CLM}{\mathbf{CLM}}
\newcommand{\IL}{\mathbf{IL}}
\newcommand{\ILP}{\mathbf{ILP}}
\newcommand{\ILM}{\mathbf{ILM}}
\newcommand{\ILW}{\mathbf{ILW}}

\newcommand{\Log}{\mathsf{Log}}

\newcommand{\J}[1]{\mathbf{J#1}}
\newcommand{\M}{\mathbf{M}}
\newcommand{\axP}{\mathbf{P}}
\newcommand{\axW}{\mathbf{W}}

\def\qed{\hfill {\small \ding{113}}}

\begin{document}

\maketitle

\begin{abstract}
We introduce and develop a topological semantics of conservativity logics and interpretability logics. 
We prove the topological compactness theorem of consistent normal extensions of the conservativity logic $\CL$ by extending Shehtman's ultrabouquet construction method to our framework. 
As a consequence, we prove that several extensions of $\CL$ such as $\IL$, $\ILM$, $\ILP$ and $\ILW$ are strongly complete with respect to our topological semantics. 
\end{abstract}


\section{Introduction}

The present paper is devoted to solving a natural problem of whether the topological semantics of the propositional modal logic $\GL$ can be extended to that of conservativity logics and interpretability logics, which are extensions of $\GL$. 
We newly introduce a topological semantics of these logics, and investigate several basic properties of our semantics such as the topological strong completeness of them.

The logic $\GL$ is known as the logic of provability (cf.~Boolos \cite{Boo93}). 
Let $\mathrm{Pr}_{\PA}(x)$ be a natural provability predicate of Peano Arithmetic $\PA$. Then, the logic $\GL$ is precisely the set of all $\PA$-verifiable modal formulas under all arithmetical interpretations where the modal operator $\Box$ is interpreted by $\mathrm{Pr}_{\PA}(x)$. 
This is called Solovay's arithmetical completeness theorem \cite{Sol76}. 
In his proof, the completeness theorem of $\GL$ with respect to Kripke semantics plays an essential role. 
Actually, it is well-known that $\GL$ is complete with respect to the class of all transitive and conversely well-founded finite Kripke frames. 
On the other hand, it is also known that $\GL$ is not strongly complete with respect to Kripke semantics, that is, there exists a set $\Gamma$ of modal formulas such that $\Gamma$ is finitely satisfiable in a transitive and conversely well-founded Kripke model, but $\Gamma$ itself is not satisfiable (See also Boolos \cite{Boo93}). 

This obstacle can be avoided by dealing with topological semantics of modal logics. 
Topological semantics of modal logic based on derived sets were initiated by McKinsey and Tarski \cite{MT44}. 
Also topological semantics of $\GL$ was founded by Simmons \cite{Sim75} and Esakia \cite{Esa81}, and has been developed by many authors (See Beklemishev and Gabelaia \cite{BG14}). 
One of important results in this research is the fact that $\GL$ is determined by the class of all scattered topological spaces. 
Moreover, as opposed to Kripke semantics, Shehtman \cite{She98} proved that $\GL$ is strongly complete with respect to scattered spaces by using so-called the method of ultrabouquet construction. 

The language of interpretability logics has the additional binary modal operator $\rhd$. 
The modal formula $\varphi \rhd \psi$ is intended to be read as ``$T + \psi$ is interpretable in $T + \varphi$", where $T$ is a suitable theory of arithmetic, such as $\PA$. 
The logic $\IL$ is a basis for the modal logical investigations of the notion of interpretability between theories, and it has been proved that the extensions $\ILM$ and $\ILP$ of $\IL$ are arithmetically complete. 
Also it is known that the notion of interpretability is closely related to that of partial conservativity. 
Actually, the logic $\ILM$ is exactly the logic of $\Pi_1$-conservativity of theories of arithmetic (See Japaridze and de Jongh \cite{JapDeJ98} for a detailed extensive survey of these results). 
From this point of view, Ignatiev \cite{Ign91} introduced the sublogic $\CL$ of $\IL$ as a basis for modal logical study of capturing properties of the notion of partial conservativity. 

A relational semantics of interpretability logics was introduced by de Jongh and Veltman \cite{deJVel90} that is called \textit{Veltman semantics}. 
A Veltman frame is a Kripke frame equipped with a family of binary relations. 
Then, de Jongh and Veltman \cite{deJVel90} proved that the logics $\IL$, $\ILM$ and $\ILP$ are complete with respect to Veltman semantics. 
Several alternative relational semantics of interpretability logics are also known, and one of important semantics was introduced by Visser \cite{Vis88} that is called \textit{simplified Veltman semantics} or \textit{Visser semantics}. 
By constructing bisimulations between corresponding Visser and Veltman frames, Visser proved that $\IL$, $\ILM$ and $\ILP$ are also complete with respect to Visser semantics. 
Moreover, Ignatiev \cite{Ign91} proved that the logic $\CL$ is complete with respect to both Veltman and Visser semantics. 
However, it can be shown that $\CL$ and $\IL$ lack strong completeness in both Veltman and Visser semantics, as in $\GL$.

On the other hand, there is a possibility of finding out the strong completeness of these logics with respect to another semantics. 
Particularly, one with respect to topological semantics is strongly suggested by Shehtman's strong completeness theorem of $\GL$. 
From this perspective, in the present paper, we propose a topological semantics of $\CL$ and its extensions, and prove the strong completeness theorem of some of these logics by extending Shehtman's method of ultrabouquet construction.

This paper is organized as follows. 
We briefly summarize Kripke and topological semantics of $\GL$ and Visser semantics of $\CL$ and its extensions in the next section. 
In Section \ref{TopSem}, we introduce a new topological semantics of normal extensions of $\CL$, and investigate some basic properties of our semantics. 
Our topological semantics is based on bitopological spaces with Visser semantics in mind. 
In Section \ref{MT}, we extend Shehtman's ultrabouquet construction to our framework, and then we prove the topological compactness theorem of consistent normal extensions of $\CL$. 
As a consequence, the topological strong completeness theorem of the logics $\CL$, $\CLM$, $\IL$, $\ILM$, $\ILP$ and $\ILW$ are obtained. 
Finally, in Section \ref{IL}, we discuss topological aspects of the logic $\IL$.

\section{Preliminaries}

The language $\mathcal{L}(\Box)$ of propositional modal logic consists of countably many propositional variables $p_0, p_1, p_2, \ldots$, logical constants $\top$, $\bot$, logical connectives $\neg, \land, \lor, \to$ and unary modal operators $\Box, \Diamond$. 
A set $L$ of $\mathcal{L}(\Box)$-formulas is said to be a \textit{normal modal logic} if $L$ contains all tautologies in the language $\mathcal{L}(\Box)$ and the formula $\Box(p \to q) \to (\Box p \to \Box q)$, and is closed under Modus Ponens $\dfrac{\varphi \to \psi\ \ \ \varphi}{\psi}$, Necessitation $\dfrac{\varphi}{\Box \varphi}$ and Substitution $\dfrac{\varphi(p_0, \ldots, p_n)}{\varphi(\psi_0, \ldots, \psi_n)}$. 
For any normal modal logic $L$, any set $\Gamma$ of $\mathcal{L}(\Box)$-formulas and any $\mathcal{L}(\Box)$-formula $\varphi$, we write $\Gamma \vdash_L \varphi$ to indicate that there exists a finite subset $\Gamma_0$ of $\Gamma$ such that $\bigwedge \Gamma_0 \to \varphi \in L$. 

The logic $\GL$ is defined as the smallest normal modal logic containing the additional axiom $\Box(\Box p \to p) \to \Box p$.

This section consists of three subsections. 
In the first subsection, we introduce Kripke semantics of $\GL$. 
The second subsection is devoted to introducing topological semantics of $\GL$, and reviewing some basic results relating to our study. 
In the last subsection, we introduce the conservativity logic $\CL$ and its extensions,  and also introduce their relational semantics, namely, Visser semantics.

\subsection{Kripke semantics of $\GL$}

\begin{defn}[Kripke frames and models]\label{KripkeFrames}\leavevmode
\begin{itemize}
	\item A pair $\langle W, R \rangle$ is said to be a \textit{Kripke frame} if $W$ is a non-empty set and $R$ is a binary relation on $W$.
	\item A triple $\langle W, R, \Vdash \rangle$ is said to be a \textit{Kripke model} if $\langle W, R \rangle$ is a Kripke frame and $\Vdash$ is a binary relation between $W$ and the set of all $\mathcal{L}(\Box)$-formulas satisfying the following conditions:
	\begin{enumerate}
		\item $x \nVdash \bot$ and $x \Vdash \top$;
		\item $x \Vdash \neg \varphi \iff x \nVdash \varphi$; 
		\item $x \Vdash \varphi \land \psi \iff x \Vdash \varphi$ and $x \Vdash \psi$; 
		\item $x \Vdash \varphi \lor \psi \iff x \Vdash \varphi$ or $x \Vdash \psi$; 
		\item $x \Vdash \varphi \to \psi \iff x \nVdash \varphi$ or $x \Vdash \psi$; 
		\item $x \Vdash \Box \varphi \iff \forall y \in W [x R y \Rightarrow y \Vdash \varphi]$; 
		\item $x \Vdash \Diamond \varphi \iff \exists y \in W [x R y\ \&\ y \Vdash \varphi]$.  
	\end{enumerate}
	\item An $\mathcal{L}(\Box)$-formula $\varphi$ is said to be \textit{valid in $\langle W, R \rangle$} if for any Kripke model $\langle W, R, \Vdash \rangle$ and any $x \in W$, $x \Vdash \varphi$. 
	\item Let $\Log(W, R)$ denote the set of all $\mathcal{L}(\Box)$-formulas valid in $\langle W, R \rangle$, and this set is called the \textit{logic of $\langle W, R \rangle$}.
\end{itemize}
\end{defn}

Notice that every $\Log(W, R)$ is a normal modal logic. 
We say that a binary relation $R$ on a set $W$ is \textit{conversely well-founded} if there is no infinite $R$-increasing sequence of elements of $W$. 
Then, the validity of the logic $\GL$ in a Kripke frame is characterized by a property of the relation $R$. 

\begin{fact}[See Boolos {\cite[Theorem 10 in Chapter 4]{Boo93}}]\label{CWF}
For any Kripke frame $\langle W, R \rangle$, $\GL \subseteq \Log(W, R)$ if and only if $R$ is transitive and conversely well-founded.
\qed
\end{fact}


We introduce the consequence relation $\models_L^K$ with respect to Kripke semantics where $K$ stands for ``Kripke''. 

\begin{defn}
Let $L$ be a normal modal logic, $\Gamma$ be a set of $\mathcal{L}(\Box)$-formulas and $\varphi$ be an $\mathcal{L}(\Box)$-formula.
\begin{itemize}
	\item $\Gamma \models_L^K \varphi$ $: \iff$ for any Kripke model $\langle W, R, \Vdash \rangle$ satisfying $L \subseteq \Log(W, R)$ and any $x \in W$, if $x \Vdash \psi$ for all $\psi \in \Gamma$, then $x \Vdash \varphi$. 
\end{itemize}
\end{defn}

Clearly, $\Gamma \vdash_L \varphi$ implies $\Gamma \models_L^K \varphi$. 
For $\GL$, the converse implication also holds in the case of $\Gamma = \varnothing$. 
This is the Kripke completeness theorem of $\GL$. 

\begin{fact}[Kripke completeness of $\GL$ (Segerberg \cite{Seg71})]\label{KCGL}
For any $\mathcal{L}(\Box)$-formula $\varphi$, $\varnothing \vdash_\GL \varphi$ if and only if $\varnothing \models_\GL^K \varphi$. 
\qed
\end{fact}

On the other hand, $\GL$ is not strongly complete with respect to Kripke semantics, that is, the equivalence of $\Gamma \vdash_{\GL} \varphi$ and $\Gamma \models_{\GL}^K \varphi$ does not hold in general. 

\begin{fact}[Fine and Rautenberg (see Boolos {\cite[pp.~102--103]{Boo93}})]\label{SKInCompl}
Let 
\[
	\Delta := \{\Diamond p_0\} \cup \{\Box (p_n \to \Diamond p_{n+1}) \mid n \in \mathbb{N}\}, 
\]
then $\Delta \models_\GL^K \bot$ but $\Delta \nvdash_\GL \bot$. 
\qed
\end{fact}


\subsection{Topological semantics of $\GL$}

For a non-empty set $X$ and a family $\tau$ of its subsets, we say that \textit{$\tau$ is a topology on $X$} if they enjoy the following conditions:
\begin{enumerate}
	\item $X, \varnothing \in \tau$; 
	\item If $U_0, U_1 \in \tau$, then $U_0 \cap U_1 \in \tau$; 
	\item For any family $\{U_i\}_{i \in I}$ of sets of $\tau$, $\bigcup_{i \in I} U_i \in \tau$. 
\end{enumerate}
Then, the pair $\langle X, \tau \rangle$ is called a \textit{topological space}. 
Every $U \in \tau$ containing $x \in X$ is called a \textit{$\tau$-neighborhood of $x$}.

\begin{defn}[Derived sets and co-derived sets]
Let $\langle X, \tau \rangle$ be a topological space and $Y \subseteq X$. 
\begin{itemize}
	\item The \textit{derived set $d_\tau(Y)$ of $Y$ (with respect to $\tau$)} is the subset of $X$ defined as follows:
\begin{align*}
 d_\tau(Y) : = \{x \in X \mid \forall U \in \tau [ x \in U \Rightarrow \exists y \neq x (y \in U \cap Y) ]\}; 
\end{align*} 
	\item The \textit{co-derived set $cd_\tau(Y)$ of $Y$ (with respect to $\tau$)} is the set $\overline{d_\tau(\overline{Y})}$, where $\overline{Y}$ is the complement of $Y$. 
\end{itemize}
\end{defn}

In topological semantics of modal logic, every topological space plays a role of a frame, and $\mathcal{L}(\Box)$-formulas are interpreted as subsets of the topological space by valuations.

\begin{defn}[Valuations on topological spaces]\label{ValTop}
Let $\langle X, \tau \rangle$ be a topological space. 
\begin{itemize}
	\item A \textit{valuation on $\langle X, \tau \rangle$} is a mapping $v: \mathcal{L}(\Box) \rightarrow \mathcal{P}(X)$ satisfying the following conditions:
\begin{enumerate}
	\item $v(\bot) = \varnothing$ and $v(\top) = X$; 
	\item $v(\neg \varphi) = \overline{v(\varphi)}$; 
	\item $v(\varphi \land \psi) = v(\varphi) \cap v(\psi)$; 
	\item $v(\varphi \lor \psi) = v(\varphi) \cup v(\psi)$; 
	\item $v(\varphi \to \psi) = \overline{v(\varphi)} \cup v(\psi)$; 
	\item $v(\Box \varphi) = cd_{\tau}(v(\varphi))$; 
	\item $v(\Diamond \varphi) = d_{\tau}(v(\varphi))$.
\end{enumerate}
	\item We say that an $\mathcal{L}(\Box)$-formula $\varphi$ is \textit{valid in $\langle X, \tau \rangle$} if $v(\varphi) = X$ for all valuations $v$ on $\langle X, \tau \rangle$.
	\item Let $\Log(X, \tau)$ be the set of all $\mathcal{L}(\Box)$-formulas valid in $\langle X, \tau \rangle$, and we call this set the \textit{logic of $\langle X, \tau \rangle$}.
\end{itemize}
\end{defn}


It is known that every $\Log(X, \tau)$ is a normal modal logic validating $p \land \Box p \to \Box \Box p$ (See Esakia \cite{Esa01} and van Benthem and Bezhanishvili \cite{vBB07}). 
As well as Fact \ref{CWF}, the validity of the logic $\GL$ in a topological space $\langle X, \tau \rangle$ is characterized by a property of $\tau$.

\begin{defn}[Scattered spaces]
A topological space $\langle X, \tau \rangle$ is said to be \textit{scattered} if for any $Y \subseteq X$, $Y \neq \varnothing$ implies $Y \setminus d_\tau(Y) \neq \varnothing$.  
\end{defn}


\begin{fact}[Simmons \cite{Sim75}; Esakia \cite{Esa81}]\label{SimEsa}
For any topological space $\langle X, \tau \rangle$, $\GL \subseteq \Log(X, \tau)$ if and only if $\langle X, \tau \rangle$ is scattered.
\qed
\end{fact}

The following fact is a summary of basic properties of derived sets.

\begin{fact}\label{dset}
Let $\langle X, \tau \rangle$ be a topological space and let $Y, Z \subseteq X$.
\begin{enumerate}
	\item $d_\tau(\varnothing) = \varnothing$; 
	\item If $Y \subseteq Z$, then $d_\tau(Y) \subseteq d_\tau(Z)$; 
	\item $d_\tau(Y \cup Z) = d_\tau(Y) \cup d_\tau(Z)$; 
	\item $Y \in \tau \iff d_\tau(\overline{Y}) \cap Y = \varnothing$; 
	\item If $\langle X, \tau \rangle$ is scattered, then $d_{\tau}(d_{\tau}(Y)) \subseteq d_{\tau}(Y)$ (cf.~\cite[Corollary 2.3]{BG14}). 
\end{enumerate}
\qed
\end{fact}

Each transitive and irreflexive Kripke frame can be considered as a topological space having the same logic via the topology of $R$-upward closed subsets.

\begin{defn}\label{upset}
Let $\langle W, R \rangle$ be a Kripke frame.
\begin{itemize}
	\item For each $x \in W$, $R(x) : = \{y \in W \mid x R y\}$; 
	\item A subset $Y \subseteq W$ is said to be \textit{$R$-upward closed} if for any $x \in Y$, $R(x) \subseteq Y$; 
	\item Define $\tau_R : = \{Y \subseteq W \mid Y$ is $R$-upward closed $\}$. 
\end{itemize}
\end{defn}

\begin{defn}[Alexandroff spaces]\label{Alexandroff}
A topological space $\langle X, \tau \rangle$ is said to be \textit{Alexandroff} if for any family $\{U_i\}_{i \in I}$ of members of $\tau$, $\bigcap_{i \in I} U_i \in \tau$.
\end{defn}

\begin{fact}[cf.~van Benthem and Bezhanishvili \cite{vBB07}]\label{KSTS}
Let $\langle W, R \rangle$ be a Kripke frame. 
Then, 
\begin{enumerate}
	\item $\langle W, \tau_R \rangle$ is an Alexandroff topological space; 
	\item If $R$ is transitive and irreflexive, then for any $Y \subseteq W$, $d_{\tau_R}(Y) = \{x \in W \mid R(x) \cap Y \neq \varnothing \}$; 
	\item If $R$ is transitive and irreflexive, then $\Log(W, R) = \Log(W, \tau_R)$. 
\end{enumerate}
\qed
\end{fact}

Alexandroff spaces will be studied precisely in Sections \ref{TopSem} and \ref{IL}. 



As in the case of Kripke semantics, we introduce the consequence relation $\models_L^T$ with respect to topological semantics where $T$ stands for ``Topology''. 

\begin{defn}
Let $L$ be a normal modal logic, $\Gamma$ be a set of $\mathcal{L}(\Box)$-formulas and $\varphi$ be an $\mathcal{L}(\Box)$-formula.
\begin{itemize}
	\item $\Gamma \models_L^T \varphi$ $:\iff$ for any topological space $\langle X, \tau \rangle$ satisfying $L \subseteq \Log(X,\tau)$, any valuation $v$ on $X$ and any $x \in X$, if $x \in v(\psi)$ for all $\psi \in \Gamma$, then $x \in v(\varphi)$. 
\end{itemize}
\end{defn}

From Facts \ref{KCGL} and \ref{KSTS}, we obtain the topological completeness of $\GL$. 

\begin{fact}[Topological completeness of $\GL$]
For any $\mathcal{L}(\Box)$-formula $\varphi$, $\varnothing \vdash_\GL \varphi$ if and only if $\varnothing \models_\GL^T \varphi$. 
\qed
\end{fact}

Moreover, as opposed to Fact \ref{SKInCompl}, Shehtman proved that $\GL$ is strongly complete with respect to topological semantics.\footnote{Actually, Shehtman proved that $\GL$ is strongly complete with respect to neighborhood semantics. 
Esakia \cite{Esa81} proved that for $\GL$, neighborhood semantics and topological semantics coincide, and so we can state Shehtman's theorem as the topological strong completeness theorem of $\GL$. }

\begin{fact}[Topological strong completeness of $\GL$ (Shehtman {\cite[Theorem 3.3]{She98}})]
Let $\Gamma$ be any set of $\mathcal{L}(\Box)$-formulas and $\varphi$ be any $\mathcal{L}(\Box)$-formula. 
Then, $\Gamma \vdash_\GL \varphi$ if and only if $\Gamma \models_\GL^T \varphi$. 
\qed
\end{fact}

\subsection{Conservativity and interpretability logics and their Visser semantics}

In this section, we introduce the conservativity logic $\CL$ and its extensions. 
Also we introduce their relational semantics. 
The language $\mathcal{L}(\Box, \rhd)$ is obtained from $\mathcal{L}(\Box)$ by adding the binary modal operator $\rhd$. 

\begin{defn}[The conservativity logic $\CL$]\label{CL}
The conservativity logic $\CL$ is a logic in the language $\mathcal{L}(\Box, \rhd)$ obtained from $\GL$ by adding the following axioms: 
\begin{description}
\item [$\J{1}$] $\Box (p \to q) \to (p \rhd q)$;
\item [$\J{2}$] $(p \rhd q) \land (q \rhd r) \to (p \rhd r)$;
\item [$\J{3}$] $(p \rhd r) \land (q \rhd r) \to ((p \lor q) \rhd r)$;
\item [$\J{4}$] $(p \rhd q) \to (\Diamond p \to \Diamond q)$.
\end{description}
\end{defn}

We say that a set $L$ of $\mathcal{L}(\Box, \rhd)$-formulas is a \textit{normal extension of $\CL$} if $\CL \subseteq L$ and $L$ is closed under Modus Ponens, Necessitation and Substitution. 
There are well-known normal extensions of $\CL$ having some of the following additional axioms: 

\begin{description}
\item [$\J{5}$] $\Diamond p \rhd p$; 
\item [$\M$] $(p \rhd q) \to ((p \land \Box r) \rhd (q \land \Box r))$;
\item [$\axP$] $(p \rhd q) \to \Box (p \rhd q)$; 
\item [$\axW$] $(p \rhd q) \to (p \rhd (q \land \Box \neg p))$. 
\end{description}

The smallest normal extension containing $\M$ is called $\CLM$. 
In this case, we write $\CLM = \CL + \M$. 
The logics $\CL$ and $\CLM$ were introduced by Ignatiev \cite{Ign91}. 
Also let $\IL = \CL + \J{5}$, $\ILM = \IL + \M$, $\ILP = \IL + \axP$ and $\ILW = \IL + \axW$. 
The logic $\IL$ is called the \textit{basic interpretability logic}. 

One of well-known relational semantics of $\CL$ and its extensions is \textit{Veltman semantics} which was introduced by de Jongh and Veltman \cite{deJVel90}. 
A triple $\langle W, R, \{S_w\}_{w \in W} \rangle$ is called a \textit{Veltman frame} if $\langle W, R \rangle$ is a transitive and conversely well-founded Kripke frame and for each $w \in W$, $S_w$ is a binary relation on $R(w)$ satisfying some additional conditions. 
One of the purposes of the present paper is to find an appropriate topological semantics of extensions of $\CL$. 
From the point of view of Fact \ref{KSTS}, every binary relation $P$ on a set $W$ is associated to the topology $\tau_P$ on $W$ consisting of $P$-upward closed subsets. 
However, each binary relation $S_w$ of Veltman frames is not a binary relation on full $W$, and so Veltman frames are not directly recognized as topological frames. 

For this reason, we adopt the alternative relational semantics of extensions of $\CL$ introduced by Visser \cite{Vis88}.

\begin{defn}[Visser frames and models]\label{VisFr}\leavevmode
\begin{itemize}
	\item A triple $\langle W, R, S \rangle$ is said to be a \textit{Visser frame} if $\langle W, R \rangle$ is a transitive and conversely well-founded Kripke frame and $S$ is a binary transitive and reflexive relation on $W$;
	\item A quadruple $\langle W, R, S, \Vdash \rangle$ is said to be a \textit{Visser model} if $\langle W, R, S \rangle$ is a Visser frame and $\Vdash$ is a binary relation as in Definition \ref{KripkeFrames} with the following additional clause: 
\begin{itemize}
		\item $x \Vdash \varphi \rhd \psi \iff \forall y\in W [x R y\ \&\ y \Vdash \varphi \Rightarrow \exists z \in W (x R z\ \&\ y S z\ \&\ z \Vdash \psi) ]$. 
\end{itemize}
	\item The validity of an $\mathcal{L}(\Box, \rhd)$-formula in Visser frames and models, and the logic $\Log(W, R, S)$ of $\langle W, R, S \rangle$ are defined as in Definition \ref{KripkeFrames}. 
\end{itemize}
\end{defn}

Visser actually introduced the notion of Visser frames as a relational semantics for extensions of $\IL$, and Definition \ref{VisFr} is an adaptation of Visser's definition to our framework obtained by removing the condition $R \subseteq S$ from his original definition. 
Visser frames are also known as simplified Veltman frames. 
Then, the following fact holds. 

\begin{fact}[See Ignatiev \cite{Ign91} and Visser \cite{Vis88}]\label{VisSem}
Let $\langle W, R, S \rangle$ be any Visser frame. 
Then, 
\begin{enumerate}
	\item $\Log(W, R, S)$ is a normal extension of $\CL$; 
	\item If $\forall x, y, z \in W[x S y R z \Rightarrow x R z]$, then $\CLM \subseteq \Log(W, R, S)$; 
	\item If $R \subseteq S$, then $\IL \subseteq \Log(W, R, S)$; 
	\item If $R \subseteq S$ and $\forall x, y, z \in W[x R y S z \Rightarrow x R z]$, then $\ILP \subseteq \Log(W, R, S)$; 
	\item If $R \subseteq S$ and the composition $R \circ S$ is conversely well-founded, then $\ILW \subseteq \Log(W, R, S)$. 
\end{enumerate}
\qed
\end{fact}

In Section \ref{IL}, we will investigate the condition $R \subseteq S$ of Visser frames from a topological viewpoint. 

We also define the consequence relation $\models_L^V$ with respect to Visser semantics. 

\begin{defn}
Let $L$ be a normal extension of $\CL$, $\Gamma$ be a set of $\mathcal{L}(\Box, \rhd)$-formulas and $\varphi$ be an $\mathcal{L}(\Box, \rhd)$-formula.
\begin{itemize}
	\item $\Gamma \models_L^V \varphi$ $:\iff$ for any Visser model $\langle W, R, S, \Vdash \rangle$ satisfying \linebreak $L \subseteq \Log(W, R, S)$ and any $x \in W$, if $x \Vdash \psi$ for all $\psi \in \Gamma$, then $x \Vdash \varphi$. 
\end{itemize}
\end{defn}

Clearly, $\Gamma \vdash_L \varphi$ implies $\Gamma \models_L^V \varphi$.
The completeness theorems of $\CL$, $\CLM$, $\IL$, $\ILP$, $\ILM$ and $\ILW$ with respect to Visser semantics are proved by Ignatiev, de Jongh and Veltman and Visser. 

\begin{fact}[Visser completeness of $\CL$ and $\CLM$ (Ignatiev \cite{Ign91})]\label{ComplCL}
Let $L \in \{\CL, \CLM\}$. 
For any $\mathcal{L}(\Box, \rhd)$-formula $\varphi$, $\varnothing \vdash_L \varphi$ if and only if $\varnothing \models_L^V \varphi$. 
\qed
\end{fact}

\begin{fact}[Visser completeness of $\IL$, $\ILM$, $\ILP$ and $\ILW$ (de Jongh and Veltman \cite{deJVel90,deJVel99} and Visser \cite{Vis88})]\label{ComplIL}
Let $L \in \{\IL, \ILM, \ILP, \ILW\}$. 
For any $\mathcal{L}(\Box, \rhd)$-formula $\varphi$, $\varnothing \vdash_L \varphi$ if and only if $\varnothing \models_L^V \varphi$. 
\qed
\end{fact}

However, every logic $L \in \{\CL, \CLM, \IL, \ILM, \ILP, \ILW\}$ lacks strong completeness with respect to Visser semantics as in the case of $\GL$. 
That is, $\Delta \models_L^V \bot$ but $\Delta \nvdash_L \bot$ where $\Delta$ is the set of formulas defined in Fact \ref{SKInCompl}.

\section{Topological semantics of normal extensions of $\CL$}\label{TopSem}

In this section, we newly introduce a topological semantics of normal extensions of $\CL$. 
Our topological semantics is based on bitopological spaces. 

\begin{defn}[Bitopological spaces]
Let $X$ be a non-empty set and $\tau^0, \tau^1$ be families of subsets of $X$. 
A triple $\langle X, \tau^0, \tau^1 \rangle$ is called a \textit{bitopological space} if both $\tau^0$ and $\tau^1$ are topologies on $X$.
\end{defn} 

The following definition is an essential part of our work. 

\begin{defn}\label{ope}
Let $\langle X, \tau^0, \tau^1 \rangle$ be a bitopological space. 
For subsets $Y$ and $Z$ of $X$, we define a subset $e_{\tau^0, \tau^1} (Y, Z)$ of $X$ as follows:
\[
	e_{\tau^0, \tau^1} (Y, Z) := \{ x \in X \mid \forall U \in \tau^1[ x \in d_{\tau^0} (Y \cap U) \Rightarrow x \in d_{\tau^0} (Z \cap U)  ]\}.
\]
\end{defn}

If there is no room for confusion, we simply write $e(Y, Z)$ instead of \linebreak $e_{\tau^0, \tau^1}(Y, Z)$. 
Using our sets $e_{\tau^0, \tau^1}(Y, Z)$, we define valuations on bitopological spaces. 

\begin{defn}
Let $\langle X, \tau^0, \tau^1 \rangle$ be a bitopological space.
A \textit{valuation} on \linebreak $\langle X, \tau^0, \tau^1 \rangle$ is a mapping $v : \mathcal{L}(\Box, \rhd) \rightarrow \mathcal{P}(X)$ defined as in Definition \ref{ValTop} with the following clauses:
\begin{itemize}
	\item $v(\Box \varphi) = cd_{\tau^0}(v(\varphi))$;   
	\item $v(\Diamond \varphi) = d_{\tau^0}(v(\varphi))$;   
	\item $v(\varphi \rhd \psi) = e_{\tau^0, \tau^1}(v(\varphi), v(\psi))$.  
\end{itemize}
The validity of an $\mathcal{L}(\Box, \rhd)$-formula in a bitopological space and the logic \linebreak $\Log(X, \tau^0, \tau^1)$ of $\langle X, \tau^0, \tau^1 \rangle$ are also defined as in Definition \ref{ValTop}. 
\end{defn}

For a normal extension $L$ of $\CL$, we say that a bitopological space $\langle X, \tau^0, \tau^1 \rangle$ is an \textit{$L$-space} if $L \subseteq \Log(X, \tau^0, \tau^1)$. 
We prove that every $\tau^0$-scattered bitopological space is a $\CL$-space.

\begin{prop}\label{AXCL}
All axioms $\J{1}$, $\J{2}$, $\J{3}$ and $\J{4}$ in Definition \ref{CL} are valid in any bitopological space $\langle X, \tau^0, \tau^1 \rangle$.
\end{prop}

\begin{proof}
($\J{1}$): 
It suffices to show that for any $Y, Z \subseteq X$, ${cd}_{\tau^0}(\overline{Y} \cup Z) \subseteq e (Y, Z)$. 
Suppose $x \in cd_{\tau^0}(\overline{Y} \cup Z)$, that is, $x \not\in d_{\tau^0} (Y \cap \overline{Z})$.
Then there exists a $\tau^0$-neighborhood $W$ of $x$ such that $Y \cap \overline{Z} \cap W \subseteq \{ x \}$. 

Take $U \in \tau^1$ arbitrarily, and suppose $x \in d_{\tau^0}(Y \cap U)$. 
We would like to show $x \in d_{\tau^0}(Z \cap U)$. 
Let $V$ be any $\tau^0$-neighborhood of $x$. 
Then $V \cap W$ is also a $\tau^0$-neighborhood of $x$. 
Since $x \in d_{\tau^0}(Y \cap U)$, there exists $y \neq x$ such that $y \in Y \cap U \cap V \cap W$, and hence $y \in Y \cap W$. 
On the other hand, since $Y \cap \overline{Z} \cap W \subseteq \{ x \}$, we have $y \not\in Y \cap \overline{Z} \cap W$. Therefore $y \in Z$, and hence $y \in Z \cap U \cap V$. 
This implies $x \in d_{\tau^0}(Z \cap U)$. 
We have shown $x \in e(Y, Z)$. 

($\J{2}$): 
We show $e(Y, Z) \cap e(Z,W) \subseteq e(Y,W)$. 
Suppose $x \in e(Y, Z) \cap e(Z,W)$. 
Take $U \in \tau^1$ arbitrarily. 
If $x \in d_{\tau^0}(Y \cap U)$, then $x \in d_{\tau^0}(Z \cap U)$ by $x \in e(Y, Z)$. Moreover, $x \in d_{\tau^0}(W \cap U)$ by $x \in e(Z,W)$. 
Thus $x \in e(Y,W)$.

($\J{3}$): 
We show $e(Y, W) \cap e(Z, W) \subseteq e(Y \cup Z, W)$. 
Suppose $x \in e(Y,W) \cap e(Z,W)$. 
Take $U \in \tau^1$ arbitrarily, and assume $x \in d_{\tau^0}((Y \cup Z) \cap U)$. 
By Fact \ref{dset}, we have
\[
	d_{\tau^0}((Y \cup Z) \cap U) = d_{\tau^0}((Y \cap U) \cup (Z \cap U)) = d_{\tau^0} (Y \cap U) \cup d_{\tau^0}(Z \cap U).
\]
Then $x \in d_{\tau^0}(Y \cap U)$ or $x \in d_{\tau^0}(Z \cap U)$. 
In either case, we obtain $x \in d_{\tau^0}(W \cap U)$ by $x \in e(Y, W) \cap e(Z, W)$. 
Thus $x \in e(Y \cup Z, W)$.

($\J{4}$): 
We show $e(Y, Z) \cap d_{\tau^0}(Y) \subseteq d_{\tau^0}(Z)$. 
Suppose $x \in e(Y, Z) \cap d_{\tau^0}(Y)$. 
Then $x \in d_{\tau^0}(Y \cap X)$. 
Since $X \in \tau^1$, it follows from $x \in e(Y, Z)$ that $x \in d_{\tau^0}(Z \cap X)$.
Equivalently, $x \in d_{\tau^0}(Z)$.
\end{proof}

Since each inference rule of $\CL$ preserves validity in bitopological spaces, we obtain the following corollary from Fact \ref{SimEsa} and Proposition \ref{AXCL}.

\begin{cor}\label{Sound_CL}
For any bitopological space $\langle X, \tau^0, \tau^1 \rangle$, it is a $\CL$-space if and only if $\langle X, \tau^0 \rangle$ is scattered. 
\qed
\end{cor}

As well as Kripke frames, Visser fames $\langle W, R, S \rangle$ can be considered as bitopological spaces by considering topologies $\tau_R$ and $\tau_S$ (see Definition \ref{upset}). 
In truth, our new operation $e_{\tau^0, \tau^1}$ is defined with the intention of satisfying the following proposition.

\begin{prop}\label{bitopval}
Let $\langle W, R, S, \Vdash \rangle$ be a Visser model. 
Let $v$ be a valuation on $\langle W, \tau_R, \tau_S \rangle$ satisfying $v(p) = \{x \in W \mid x \Vdash p\}$ for any propositional variable $p$, then $v(\varphi) = \{x \in W \mid x \Vdash \varphi\}$ for any $\mathcal{L}(\Box, \rhd)$-formula $\varphi$.
\end{prop}
\begin{proof}
We prove by induction on the construction of $\varphi$. 
We provide proofs of only two cases that $\varphi$ is $\Diamond \psi$ and $\varphi$ is $\psi \rhd \chi$. 

Case of $\varphi \equiv \Diamond \psi$: 
\begin{align*}
	x \Vdash \Diamond \psi & \iff \exists y \in W(x R y \ \&\ y \Vdash \psi), \\
	& \iff R(x) \cap v(\psi) \neq \varnothing, \tag{by induction hypothesis}\\
	& \iff x \in d_{\tau_R}(v(\psi)), \tag{by Fact \ref{KSTS}.2}\\
	& \iff x \in v(\Diamond \psi). 
\end{align*}

Case of $\varphi \equiv \psi \rhd \chi$: 
\begin{align*}
	x \Vdash \psi \rhd \chi & \iff \forall y [ x R y \ \&\ y \Vdash \psi \Rightarrow \exists z(x R z\ \&\ y S z\ \&\ z \Vdash \chi) ], \\
	& \iff \forall y [ y \in R(x) \cap v(\psi) \Rightarrow R(x) \cap S(y) \cap v(\chi) \neq \varnothing ], \tag{by induction hypothesis}\\
	& \overset{(\ast)}{\iff} \forall U \in \tau_S [ R(x) \cap v(\psi) \cap U \neq \varnothing \Rightarrow R(x) \cap U \cap v(\chi) \neq \varnothing ], \\
	& \iff \forall U \in \tau_S [ x \in d_{\tau_R}(v(\psi) \cap U) \Rightarrow x \in d_{\tau_R}(U \cap v(\chi)) ], \tag{by Fact \ref{KSTS}.2}\\
	& \iff x \in e_{\tau_R, \tau_S}(v(\psi), v(\chi)), \\
	& \iff x \in v(\psi \rhd \chi). 
\end{align*}

Here we give a proof of the equivalence marked by $(\ast)$. 

$(\Rightarrow)$: 
Let $U$ be any element of $\tau_S$ with $R(x) \cap v(\psi) \cap U \neq \varnothing$. 
Let $y \in R(x) \cap v(\psi) \cap U$. 
Then, $R(x) \cap S(y) \cap v(\chi)$ is non-empty. 
Since $U$ is $S$-upward closed, $S(y) \subseteq U$. 
Thus $R(x) \cap U \cap v(\chi)$ is also non-empty. 

$(\Leftarrow)$: 
Let $y$ be any element of $R(x) \cap v(\psi)$. 
Since $S$ is reflexive, $y \in S(y)$, and hence $y \in R(x) \cap v(\psi) \cap S(y)$. 
It follows from the transitivity of $S$ that $S(y)$ is $S$-upward closed. 
Hence $S(y) \in \tau_S$. 
Then, we obtain that $R(x) \cap S(y) \cap v(\chi)$ is non-empty. 
\end{proof}

From Proposition \ref{bitopval}, we obtain the following corollary. 

\begin{cor}\label{VSTS}
For any Visser frame $\langle W, R, S \rangle$, $\Log(W, R, S) = \Log(W, \tau_R, \tau_S)$. 

\qed
\end{cor}

Since every transitive and conversely well-founded Kripke frame can be extended to a Visser frame, Corollary \ref{VSTS} is an extension of Fact \ref{KSTS}.3. 
Conversely, we show that $\tau^0$-scattered Alexandroff bitopological spaces can be considered as Visser frames.

\begin{thm}\label{Alex2}
Let $\langle X, \tau^0, \tau^1 \rangle$ be any bitopological space. 
Then, the following are equivalent:
\begin{enumerate}
	\item $\tau^0$ is scattered and both $\tau^0$ and $\tau^1$ are Alexandroff. 
	\item There exists a Visser frame $\langle X, R, S \rangle$ such that $\tau^0 = \tau_R$ and $\tau^1 = \tau_S$.
\end{enumerate} 
\end{thm}
\begin{proof}
$(\Rightarrow)$: 
We define binary relations $R$ and $S$ on $X$ as follows: 
\begin{itemize}
	\item $x R y : \iff x \neq y\ \&\ \forall U \in \tau^0 (x \in U \Rightarrow y \in U)$\\
	\hspace{0.3in} $(\iff x \in d_{\tau^0}(\{y\}))$; 
	\item $x S y : \iff \forall U \in \tau^1 (x \in U \Rightarrow y \in U)$. 
\end{itemize}
Clearly, $R$ is irreflexive and $S$ is transitive and reflexive. 
We show that $R$ is transitive. 
Let $x R y$ and $y R z$. 
Then $x \in d_{\tau^0}(\{y\})$ and $y \in d_{\tau^0}(\{z\})$. 
By Fact \ref{dset}.2, $d_{\tau^0}(\{y\}) \subseteq d_{\tau^0}(d_{\tau^0}(\{z\}))$. 
Since $\tau^0$ is scattered, $d_{\tau^0}(d_{\tau^0}(\{z\})) \subseteq d_{\tau^0}(\{z\})$ by Fact \ref{dset}.5. 
Thus $d_{\tau^0}(\{y\}) \subseteq d_{\tau^0}(\{z\})$. 
Then, $x \in d_{\tau^0}(\{z\})$ and hence $x R z$. 

We prove $\tau^0 = \tau_R$, and the proof of $\tau^1 = \tau_S$ is similar. 

$(\subseteq)$: 
Let $U \in \tau^0$. 
If $x \in U$ and $x R y$, then $y \in U$ by the definition of $R$. 
This means that $U$ is $R$-upward closed. 
Thus $U \in \tau_R$. 

$(\supseteq)$: Let $U \in \tau_R$ and $x$ be an arbitrary element of $U$. 
Define $V' : = \bigcap \{V \in \tau^0 \mid x \in V\}$. 
Since $\tau^0$ is Alexandroff, $V'$ is a $\tau^0$-neighborhood of $x$. 
Since $V'$ is a subset of every $\tau^0$-neighborhood of $x$, for any $y \in V'$, either $x = y$ or $x R y$. 
Since $U$ is $R$-upward closed, $U$ contains such $y$. 
Thus $V' \subseteq U$. 
We have shown that an arbitrary element of $U$ has a $\tau^0$-neighborhood inside of $U$. 
Thus $U \in \tau^0$. 

Since $\langle X, \tau^0 \rangle$ is scattered, by Fact \ref{SimEsa}, $\GL \subseteq \Log(X, \tau^0)$. 
By Fact \ref{KSTS}.3, $\Log(X, R) = \Log(X, \tau_R) = \Log(X, \tau^0)$. 
Then $\GL \subseteq \Log(X, R)$, and thus $R$ is conversely well-founded by Fact \ref{CWF}. 
Therefore $\langle W, R, S \rangle$ is a Visser frame. 

$(\Leftarrow)$: 
By Fact \ref{KSTS}.1, both $\tau^0 = \tau_R$ and $\tau^1 = \tau_S$ are Alexandroff. 
Since $R$ is transitive and conversely well-founded, $\GL \subseteq \Log(W, R) = \Log(W, \tau_R)$ by Facts \ref{CWF} and \ref{KSTS}.3. 
Then it follows from Fact \ref{SimEsa} that $\tau^0 = \tau_R$ is scattered. 
\end{proof}

To summarize the previous investigations, Visser semantics is exactly a topological semantics restricted to $\tau^0$-scattered Alexandroff bitopological spaces. 
Some extensions of $\CL$ such as $\IL$ are complete but not strongly complete with respect to this restricted version of topological semantics.

As in the previous section, we introduce the consequence relation $\models_L^T$ with respect to our topological semantics. 

\begin{defn}
Let $L$ be a normal extension of $\CL$, $\Gamma$ be a set of $\mathcal{L}(\Box, \rhd)$-formulas, and $\varphi$ be an $\mathcal{L}(\Box, \rhd)$-formula.
\begin{itemize}
	\item $\Gamma \models_L^T \varphi$ $:\iff$ for any $L$-space $\langle X, \tau^0, \tau^1 \rangle$, any valuation $v$ on $\langle X, \tau^0, \tau^1 \rangle$ and any $x \in X$, if $x \in v (\psi)$ for all $\psi \in \Gamma$, then $x \in v (\varphi)$; 
	\item We say that $L$ is \textit{topologically complete} if for any $\mathcal{L}(\Box, \rhd)$-formula $\varphi$, $\varnothing \models_L^T \varphi$ implies $\varnothing \vdash_L \varphi$; 
	\item We say that $L$ is \textit{topologically strongly complete} if for any $\mathcal{L}(\Box, \rhd)$-formula $\varphi$ and set $\Gamma$ of $\mathcal{L}(\Box, \rhd)$-formulas, $\Gamma \models_L^T \varphi$ implies $\Gamma \vdash_L \varphi$. 
\end{itemize}
\end{defn}

From Facts \ref{ComplCL} and \ref{ComplIL}, and the above discussions, we obtain the following topological completeness of $\CL$ and its some extensions.  

\begin{thm}[Topological completeness of some extensions of $\CL$]\label{wTCCL}
The logics $\CL$, $\CLM$, $\IL$, $\ILM$, $\ILP$ and $\ILW$ are topologically complete. 
\qed
\end{thm}

The main purpose of the present paper is to strengthen Theorem \ref{wTCCL}, that is, we prove that these logics are topologically strongly complete.

\section{Topological compactness and topological strong completeness}\label{MT}

In this section, we prove the topological strong completeness theorem of some extensions of $\CL$. 
This directly follows from the the topological compactness theorem (Theorem \ref{CompCL}) and the topological completeness theorem (Theorem \ref{wTCCL}). 
Thus the main purpose of this section is to prove the topological compactness theorem. 
We prove this theorem by extending the method of Shehtman's ultrabouquet construction for topological spaces (cf.~Shehtman \cite{She98,She05}) to our framework. 

\subsection{The ultrabouquet construction for bitopological spaces}

We introduce the notion of the ultrabouquet of a countable family \linebreak $\{ \langle X_n, \tau_n^0, \tau_n^1 \rangle\}_{n \in \mathbb{N}}$ of bitopological spaces, and investigate properties of ultrabouquets used in our proof of the topological compactness theorem. 
Before introducing it, we recall the following fact. 

\begin{fact}[cf.~Shehtman {\cite[Lemma 61]{She05}}]\label{FactSc}
Let $\langle X, \tau \rangle$ be a scattered space. Then for any $x \in X$, there exists $Y \subseteq X$ such that $Y$ is a $\tau$-neighborhood of $x$ and $Y \setminus \{ x \} \in \tau$.
\end{fact}

In this subsection, we fix a countable family $\{ \langle X_n, \tau_n^0, \tau_n^1 \rangle \}_{n \in \mathbb{N}}$ of bitopological spaces satisfying the following conditions:
\begin{itemize}
	\item All topological spaces $\langle X_n, \tau^0_n \rangle$ are scattered; 
	\item The family $\{X_n\}_{n \in \mathbb{N}}$ is pairwise disjoint. 
\end{itemize}
We also fix a family $\{x_n\}_{n \in \mathbb{N}}$ of elements such that $x_n \in X_n$ for every $n \in \mathbb{N}$. 
Then by Fact \ref{FactSc}, for each $n \in \mathbb{N}$, there exists $Y_n \subseteq X_n$ such that $Y_n$ is $\tau_n^0$-neighborhood of $x_n$ and $Y_n \setminus \{ x_n \} \in \tau_n^0$.
Let $\mathcal{U}$ be a non-principal ultrafilter on $\mathbb{N}$. 
Let $x_\ast$ be a new element not contained in $\bigcup_{n \in \mathbb{N}} X_n$. 

\begin{defn}\label{UB}
We define an \textit{ultrabouquet} $\mathfrak{X} := \langle X, \tau^0, \tau^1 \rangle$ based on the families $\{ \langle X_n, \tau_n^0, \tau_n^1 \rangle\}_{n \in \mathbb{N}}$ and $\{x_n\}_{n \in \mathbb{N}}$ as follows: 
\begin{itemize}
\item $X:= \bigcup_{n \in \mathbb{N}} (X_n \setminus \{ x_n \} ) \cup \{x_\ast \}$. 

For each $V \subseteq X$ and $n \in \mathbb{N}$, we sometimes restrict $V$ to $X_n$ or $Y_n$. 
In these situations, we would like to identify $x_\ast$ with $x_n$.
From this perspective, we let:
\[
V \restriction X_n :=
\begin{cases}
V \cap X_n & \text{ if } x_\ast \not\in V; \\
((V \setminus \{ x_\ast \}) \cup \{ x_n \})\cap X_n  & \text{ if } x_\ast \in V. 
\end{cases}
\]
Also $V \restriction Y_n$ is defined in a similar way. 

\item $U \in \tau^0 : \iff $
	\begin{itemize}
	\item [(i)] For each $n \in \mathbb{N}$, $U \cap (Y_n \setminus \{ x_n \}) \in \tau_n^0$;  and
	\item [(ii)] If $x_\ast \in U$, then $\{ n \in \mathbb{N} \mid U \restriction Y_n \in \tau_n^0 \} \in \mathcal{U}$.
	\end{itemize}
\item $U \in \tau^1 : \iff$ for each $n \in \mathbb{N}$, $U \restriction X_n \in \tau_n^1$.
\end{itemize}
\end{defn}

\begin{lem}
The ultrabouquet $\mathfrak{X}$ is a bitopological space. 
\end{lem}

\begin{proof}
We only prove that $\tau^0$ is a topology on $X$. 
A proof for $\tau^1$ is similar. 

\begin{itemize}
	\item $\varnothing \in \tau^0$: (i) $\varnothing \cap (Y_n \setminus \{x_n\}) = \varnothing \in \tau_n^0$; and (ii) $x_\ast \notin \varnothing$.
 
	\item $X \in \tau^0$: (i) $X \cap (Y_n \setminus \{x_n\}) = Y_n \setminus \{x_n\} \in \tau_n^0$; and (ii) Since $X \restriction Y_n = Y_n \in \tau_n^0$, $\{n \in \mathbb{N} \mid X \restriction Y_n \in \tau_n^0\} = \mathbb{N} \in \mathcal{U}$ because $\mathcal{U}$ is a non-trivial filter. 

	\item Let $U_0, U_1 \in \tau^0$. 
We show $U_0 \cap U_1 \in \tau^0$. 
(i): By condition (i) for $U_0$ and $U_1$, the sets $U_0 \cap (Y_n \setminus \{x_n\})$ and $U_1 \cap (Y_n \setminus \{x_n\})$ are elements of $\tau_n^0$. 
Then 
\[
	(U_0 \cap U_1) \cap (Y_n \setminus \{x_n\}) = (U_0 \cap (Y_n \setminus \{x_n\})) \cap (U_1 \cap (Y_n \setminus \{x_n\})) \in \tau_n^0. 
\]

(ii): If $x_\ast \in U_0 \cap U_1$, then $x_\ast$ is in both $U_0$ and $U_1$. 
By condition (ii) for $U_0$ and $U_1$, the sets $Z_0 = \{ n \in \mathbb{N} \mid U_0 \restriction Y_n \in \tau_n^0 \}$ and $Z_1 = \{ n \in \mathbb{N} \mid U_1 \restriction Y_n \in \tau_n^0 \}$ are in $\mathcal{U}$. 
Then, 
\[
	Z_0 \cap Z_1 \subseteq \{ n \in \mathbb{N} \mid (U_0 \cap U_1) \restriction Y_n \in \tau_n^0 \} \in \mathcal{U}
\]
because $\mathcal{U}$ is a filter. 

	\item $\{U_i\}_{i \in I}$ be any family of elements of $\tau^0$. 
	We show $\bigcup_{i \in I} U_i \in \tau^0$. 
(i): Since $U_i \cap (Y_n \setminus \{x_n\}) \in \tau_n^0$ for all $i \in I$,
\[
	\left(\bigcup_{i \in U} U_i \right) \cap (Y_n \setminus \{x_n\}) = \bigcup_{i \in U} \left(U_i \cap (Y_n \setminus \{x_n\}) \right) \in \tau_n^0.
\]

(ii): If $x_\ast \in \bigcup_{i \in I} U_i$, then $x_\ast \in U_j$ for some $j \in I$. 
By condition (ii) for $U_j$, $\{n \in \mathbb{N} \mid U_j \restriction Y_n \in \tau_n^0\} \in \mathcal{U}$. 

\begin{cl}\label{Cl1}
For $n \in \mathbb{N}$, if $U_j \restriction Y_n \in \tau_n^0$, then $(\bigcup_{i \in I}U_i) \restriction Y_n \in \tau_n^0$. 
\end{cl}
\begin{proof}[Proof of Claim \ref{Cl1}]
Let $x$ be an arbitrary element of $(\bigcup_{i \in I}U_i) \restriction Y_n$. 
We show that there exists a $\tau_n^0$-neighborhood $V$ of $x$ satisfying $V \subseteq (\bigcup_{i \in I}U_i) \restriction Y_n$. 
We distinguish the following two cases: 

If $x = x_n$, then $U_j \restriction Y_n$ is a required $\tau_n^0$-neighborhood of $x$. 

If $x \neq x_n$, then $x \in U_k \cap (Y_n \setminus \{x_n\})$ for some $k \in I$. 
	By condition (i) for $U_k$, this set is a required $\tau_n^0$-neighborhood of $x$. 

Therefore $(\bigcup_{i \in I}U_i) \restriction Y_n \in \tau_n^0$. 
\end{proof}

From Claim \ref{Cl1}, we have 
\[
	\{n \in \mathbb{N} \mid U_j \restriction Y_n \in \tau_n^0\} \subseteq \{n \in \mathbb{N} \mid (\bigcup_{i \in I} U_i) \restriction Y_n \in \tau_n^0\} \in \mathcal{U}. 
\]
\end{itemize}

\end{proof}

For each $n \in \mathbb{N}$, let $v_n$ be a valuation on $\langle X_n, \tau_n^0, \tau_n^1 \rangle$. 
We define a valuation $v$ on $\mathfrak{X}$ as follows:

\begin{defn}\label{bouquetval}\leavevmode
\begin{itemize}
\item For $x \in X_n \setminus \{x_n\}$, $x \in v (p) : \iff x \in v_n (p)$;
\item $x_\ast \in v (p) :\iff \{ n \in \mathbb{N} \mid x_n \in v_n (p) \} \in \mathcal{U}$.
\end{itemize}
\end{defn}

Let $Y$ denote the set $\bigcup_{n \in \mathbb{N}} (Y_n \setminus \{ x_n \}) \cup \{ x_\ast \}$.
We investigate the images of the valuation $v$ by dividing $X$ into three parts, namely, $Y \setminus \{x_\ast\}$, $X \setminus Y$ and $\{x_\ast\}$. 

First, we investigate in $Y \setminus \{x_\ast\}$. 
If $x \in Y \setminus \{x_\ast\}$, then $x$ is in $Y_n \setminus \{x_n\}$ for some $n \in \mathbb{N}$. 
In the set $Y_n \setminus \{ x_n \}$, the first clause of Definition \ref{bouquetval} is extended to all $\mathcal{L}(\Box, \rhd)$-formulas as follows. 

\begin{lem}\label{TruthL}
For any $\mathcal{L}(\Box, \rhd)$-formula $\varphi$, $n \in \mathbb{N}$ and $x \in Y_n \setminus \{x_n\}$, 
\[
	x \in v (\varphi) \iff x \in v_n (\varphi). 
\]
\end{lem}

\begin{proof}
We prove by induction on the construction of $\varphi$. 
We only give a proof of the case $\varphi \equiv \psi \rhd \chi$.

($\Rightarrow$): 
Suppose $x \in v(\psi \rhd \chi)$. 
Then
\begin{align}\label{p1}
 \forall U \in \tau^1 \left[ x \in d_{\tau^0} \left(v (\psi) \cap U \right) \Rightarrow x \in d_{\tau^0} \left( v(\chi) \cap U \right) \right].
\end{align}
In order to prove $x \in v_n (\psi \rhd \chi)$, let $U$ be an arbitrary element of $\tau_n^1$ and assume $x \in d_{\tau_n^0} (v_n (\psi) \cap U)$. 
We would like to show $x \in d_{\tau_n^0} (v_n (\chi) \cap U)$. 
Let 
\[
	U' : = \begin{cases} U & \text{if}\ x_n \notin U; \\ ((U \setminus \{x_n\}) \cup \bigcup_{m \neq n} X_m \setminus \{x_m\}) \cup \{x_\ast\} & \text{if}\ x_n \in U. \end{cases}
\]
Then, it is easily shown that $U' \in \tau^1$ and $U' \restriction X_n = U$.

\begin{cl}\label{Cl2}
$x \in d_{\tau^0} \left( v(\psi) \cap U' \right)$.
\end{cl}
\begin{proof}[Proof of Claim \ref{Cl2}]
Let $V$ be any $\tau^0$-neighborhood of $x$. 
By Definition \ref{UB}, $V \cap ( Y_n \setminus \{ x_n \} ) \in \tau_n^0$, and hence the set $V \cap ( Y_n \setminus \{ x_n \} )$ is a $\tau_n^0$-neighborhood of $x$. 
Since $x \in d_{\tau_n^0} (v_n (\psi) \cap U)$, there exists $y \neq x$ such that $y \in v_n (\psi) \cap U \cap V \cap ( Y_n \setminus \{ x_n \} )$. 
By the induction hypothesis, $y \in v (\psi) \cap U \cap V \cap (Y_n \setminus \{ x_n \})$. 
Hence $y \in v(\psi) \cap U' \cap V$. 
This implies $x \in d_{\tau^0} \left( v(\psi) \cap U' \right)$.
\end{proof}
From (\ref{p1}) and Claim \ref{Cl2}, we have $x \in d_{\tau^0} ( v(\chi) \cap U')$.

\begin{cl}\label{Cl3}
$x \in d_{\tau_n^0}(v_n (\chi) \cap U)$.
\end{cl}
\begin{proof}[Proof of Claim \ref{Cl3}]
Let $V$ be any $\tau_n^0$-neighborhood of $x$. 
Then, $V \cap (Y_n \setminus \{ x_n \} ) \in \tau_n^0$ and $x \in V \cap (Y_n \setminus \{ x_n \})$. 
Together with $x_\ast \not\in V \cap (Y_n \setminus \{ x_n \} )$, it is shown that the set $V \cap (Y_n \setminus \{ x_n \})$ is a $\tau^0$-neighborhood of $x$. 
Since $x \in d_{\tau^0} ( v(\chi) \cap U' )$, there exists $y \neq x$ such that $y \in v(\chi) \cap U' \cap V \cap (Y_n \setminus \{ x_n \})$. 
By the induction hypothesis, $y \in v_n (\chi) \cap U' \cap V \cap (Y_n \setminus \{ x_n \})$. 
Since $U' \restriction X_n = U$, we conclude $y \in v_n(\chi) \cap U \cap V$. 
\end{proof}
We have shown $x \in e_{\tau_n^0, \tau_n^1}(v_n(\psi), v_n(\chi)) = v_n(\psi \rhd \chi)$.

\vspace{0.1in}

($\Leftarrow$): 
Suppose $x \in v_n (\psi \rhd \chi)$. 
Then
\begin{align}\label{p2}
 \forall U \in \tau_n^1 \left[ x \in d_{\tau_n^0}(v_n (\psi) \cap U) \Rightarrow x \in d_{\tau_n^0}( v_n(\chi) \cap U) \right].
\end{align}
Let $U$ be an arbitrary element of $\tau^1$ and assume $x \in d_{\tau^0} (v(\psi) \cap U)$. 
We would like to show $x \in d_{\tau^0} (v(\chi) \cap U)$. 
Let $U' := U \restriction X_n$, then $U' \in \tau_n^1$.

\begin{cl}\label{Cl4}
$x \in d_{\tau_n^0}( v_n(\psi) \cap U')$.
\end{cl}
\begin{proof}[Proof of Claim \ref{Cl4}]
Let $V$ be any $\tau_n^0$-neighborhood of $x$. 
Then $V \cap (Y_n \setminus \{ x_n \}) \in \tau_n^0$ and $x \in V \cap (Y_n \setminus \{ x_n \})$. 
Together with $x_\ast \notin V \cap (Y_n \setminus \{ x_n \})$,
it is shown that the set $V \cap ( Y_n \setminus \{ x_n \})$ is a $\tau^0$-neighborhood of $x$. 
Since $x \in d_{\tau^0} (v(\psi) \cap U)$, there exists $y \neq x$ such that $y \in v(\psi) \cap U \cap V \cap (Y_n \setminus \{x_n\})$. By the induction hypothesis, $y \in v_n(\psi) \cap U \cap V \cap (Y_n\backslash\{x_n\})$, and hence $y \in v_n(\psi) \cap U' \cap V$. 
Thus we conclude $x \in d_{\tau_n^0}(v_n(\psi) \cap U')$.
\end{proof}

From (\ref{p2}) and Claim \ref{Cl4}, $x \in d_{\tau_n^0}(v_n(\chi) \cap U')$.

\begin{cl}\label{Cl5}
$x \in d_{\tau^0} (v(\chi) \cap U)$.
\end{cl}
\begin{proof}[Proof of Claim \ref{Cl5}]
Let $V$ be any $\tau^0$-neighborhood of $x$. 
By Definition \ref{UB}, $V \cap (Y_n \setminus \{x_n\}) \in \tau_n^0$ and hence $V \cap (Y_n \setminus \{x_n\})$ is a $\tau_n^0$-neighborhood of $x$. 
Since $x \in d_{\tau_n^0}(v_n(\chi) \cap U')$, there exists $y \neq x$ such that $y \in v_n(\chi) \cap U' \cap V \cap (Y_n \setminus \{x_n\})$. 
By the induction hypothesis, $y \in v(\chi) \cap U' \cap V \cap (Y_n \setminus \{x_n\})$, and hence $y \in v(\chi) \cap U \cap V$. 
Thus we conclude $x \in d_{\tau^0}(v (\chi) \cap U)$.
\end{proof}

We have proved $x \in e_{\tau^0, \tau^1}(v(\psi), v(\chi)) = v(\psi \rhd \chi)$.
This completes our proof of Lemma \ref{TruthL}. 
\end{proof}

Secondly, we investigate the behavior of valuations on $\mathfrak{X}$ in $X \setminus Y$. 

\begin{lem}\label{LIso}
For any subset $U$ of $X \setminus Y$, $U \in \tau^0$.
\end{lem}
\begin{proof}
We show that each $U \subseteq X \setminus Y$ satisfies conditions (i) and (ii) in Definition \ref{UB}. 
Clearly $U \cap ( Y_n \setminus \{ x_n \} ) = \varnothing$ for any $n \in \mathbb{N}$, and hence (i) holds. 
Moreover, (ii) vacuously holds since $U$ does not contain $x_\ast$.
\end{proof}

The following lemma shows that every element of $X \setminus Y$ behaves as a dead end of Kripke frames. 

\begin{lem}\label{deadend}
For any $x \in X \setminus Y$ and any $Z \subseteq X$, $x \in cd_{\tau^0}(Z)$. 
\end{lem}
\begin{proof}
Let $x \in X \setminus Y$. 
Then, by Lemma \ref{LIso}, $\{x\} \in \tau^0$. 
Since $\overline{Z} \cap \{x\} \subseteq \{x\}$, we have $x \notin d_{\tau^0}(\overline{Z})$. 
That is, $x \in cd_{\tau^0}(Z)$. 
\end{proof}

For $x \in X_n \setminus Y_n$, even if $x \in v_n(\Diamond \varphi)$, by Lemma \ref{deadend}, $x \notin v(\Diamond \varphi)$. 
So the equivalence of Lemma \ref{TruthL} cannot be extended to elements of $X_n \setminus \{x_n\}$. 

Thirdly, the following lemma is a generalization of the second clause of Definition \ref{bouquetval}. 
In particular, it plays a key role in our proof of the topological compactness theorem. 

\begin{lem}\label{TruthL2}
For any $\mathcal{L}(\Box, \rhd)$-formula $\varphi$, 
\[
	x_\ast \in v(\varphi) \iff \{ n \in \mathbb{N} \mid x_n \in v_n(\varphi) \} \in \mathcal{U}.
\]
\end{lem}

\begin{proof}
We prove by induction on the construction of $\varphi$. 
We only give a proof of the case $\varphi \equiv \psi \rhd \chi$.

($\Rightarrow$): 
We prove the contrapositive. 
Assume $\{ n \in \mathbb{N} \mid x_n \in v_n(\psi \rhd \chi) \} \not\in \mathcal{U}$. Since $\mathcal{U}$ is an ultrafilter on $\mathbb{N}$, $Z_0: = \{ n \in \mathbb{N} \mid x_n \not\in v_n(\psi \rhd \chi) \} \in \mathcal{U}$. 
For each $n \in Z_0$, there exists $U_n \in \tau_n^1$ such that
\begin{align}\label{p3}
	x_n \in d_{\tau_n^0}(v_n(\psi) \cap U_n)\ \& \ x_n \not\in d_{\tau_n^0}(v_n(\chi) \cap U_n).
\end{align}
Let $Z_{00} : = \{n \in Z_0 \mid x_n \notin U_n\}$ and $Z_{01} : = \{n \in Z_0 \mid x_n \in U_n\}$. 
Then, $Z_0 = Z_{00} \cup Z_{01}$. 
Since $\mathcal{U}$ is an ultrafilter, we get an $i \in \{0, 1\}$ such that $Z_{0i} \in \mathcal{U}$. 
Let 
\[
	U : = \begin{cases} \bigcup_{n \in Z_{0i}} U_n & \text{if}\ i = 0; \\ \left(\bigcup_{n \in Z_{0i}} U_n \setminus \{x_n\} \right) \cup \left( \bigcup_{n \notin Z_{0i}} X_n \setminus \{x_n\} \right) \cup \{x_\ast\} & \text{if}\ i = 1. \end{cases}
\]
Then, it is shown that $U$ is an element of $\tau^1$ satisfying $U \restriction X_n = U_n$ for all $n \in Z_{0i}$. 

First, we prove $x_\ast \in d_{\tau^0}(v(\psi) \cap U)$. 
Let $V$ be any $\tau^0$-neighborhood of $x_\ast$. 
By Definition \ref{UB}, $Z_1:= \{ n\in\mathbb{N}\mid V \restriction Y_n \in \tau_n^0 \} \in \mathcal{U}$. 
Since $Z_{0i} \cap Z_1 \in \mathcal{U}$, $Z_{0i} \cap Z_1$ is non-empty, and fix some $n \in Z_{0i} \cap Z_1$. 
Since the set $V \restriction Y_n$ is a $\tau_n^0$-neighborhood of $x_n$, by (\ref{p3}), there exists $y \in X_n \setminus\{ x_n\}$ such that $y \in v_n(\psi) \cap U_n \cap V \cap (Y_n \setminus \{ x_n\})$. 
Applying Lemma \ref{TruthL}, $y \in v(\psi) \cap U_n \cap V \cap (Y_n \setminus \{ x_n \})$. 
Since $U_n = U \restriction X_n$, we obtain $y \in v(\psi) \cap U \cap V$ and $y \neq x_\ast$. 
Thus $x_\ast \in d_{\tau^0}(v(\psi) \cap U)$.

Secondly, we prove $x_\ast \not\in d_{\tau^0} (v(\chi) \cap U)$. 
By (\ref{p3}), for each $n \in Z_{0i}$, there exists a $\tau_n^0$-neighborhood $W_n$ of $x_n$ such that $v_n(\chi) \cap U_n \cap W_n \subseteq \{ x_n \}$. 
Let $W := \bigcup_{n \in Z_{0i}} (W_n \cap (Y_n \setminus \{x_n\}) \cup \{x_\ast\}$. 
We show $W \in \tau^0$. 
(i) For each $n \in \mathbb{N}$,
\[
 W \cap (Y_n \setminus \{ x_n\}) =
\begin{cases}
W_n \cap (Y_n \setminus \{ x_n \}) & \text{if} \ n \in Z_{0i}; 
\\
\varnothing & \text{otherwise.}
\end{cases}
\]
Then, $W \cap (Y_n \setminus \{ x_n\}) \in \tau_n^0$. 
(ii) If $n \in Z_{0i}$, then $W \restriction Y_n = W_n \cap Y_n \in \tau_n^0$. 
Hence $Z_{0i} \subseteq \{ n \in \mathbb{N} \mid W \restriction Y_n \in \tau_n^0 \} \in \mathcal{U}$ because $\mathcal{U}$ is a filter. 
Thus $W$ is a $\tau^0$-neighborhood of $x_\ast$. 

Suppose, towards a contradiction, that $x_\ast \in d_{\tau^0}(v(\chi) \cap U)$. 
Then there exists $y \neq x_\ast$ such that $y \in v(\chi) \cap U \cap W$. 
Since $y \in W$, for some $n \in Z_{0i}$, $y \in v(\chi) \cap U \cap W_n \cap (Y_n \setminus \{ x_n \})$. 
Applying Lemma \ref{TruthL}, $y \in v_n(\chi) \cap U \cap W_n \cap (Y_n \setminus \{ x_n \})$. 
Since $U \restriction X_n = U_n$, $y \in v_n(\chi) \cap U_n \cap W_n$. 
This contradicts $v_n(\chi) \cap U_n \cap W_n \subseteq \{ x_n \}$.
Therefore $x_\ast \notin d_{\tau^0}(v(\chi) \cap U)$. 

We conclude $x_\ast \not\in e_{\tau^0, \tau^1}(v(\psi), v(\chi))$, and hence $x_\ast \not\in v(\psi \rhd \chi)$.

\vspace{0.1in}

($\Leftarrow$):
Suppose $Z_0 := \{ n \in \mathbb{N} \mid x_n \in v_n(\psi \rhd \chi) \} \in \mathcal{U}$. In order to prove $x_\ast \in v(\psi \rhd \chi)$, suppose that $U \in \tau^1$ and $x_\ast \in d_{\tau^0} (v(\psi) \cap U)$. 
We would like to show $x_\ast \in d_{\tau^0}(v(\chi) \cap U)$. 
Let $V$ be any $\tau^0$-neighborhood of $x_\ast$. By Definition \ref{UB}, $Z_1: = \{ n \in \mathbb{N} \mid V \restriction Y_n \in \tau_n^0 \} \in \mathcal{U}$. 
For each $n \in \mathbb{N}$, let $U_n : = U \restriction X_n$. 
Then $U_n \in \tau_n^1$.

\begin{cl}\label{Cl6}
There exists $n \in Z_0 \cap Z_1$ such that $x_n \in d_{\tau_n^0} (v_n (\psi) \cap U_n)$. 
\end{cl}
\begin{proof}[Proof of Claim \ref{Cl6}]
Suppose, towards a contradiction, that for all $n \in Z_0 \cap Z_1$, $x_n \not\in d_{\tau_n^0} (v_n(\psi) \cap U_n )$. 
Then for each $n \in Z_0 \cap Z_1$, there exists $W_n \in \tau_n^0$ such that $x_n \in W_n$ and $v_n(\psi) \cap U_n \cap W_n \subseteq \{ x_n \}$. 
Let $W := \bigcup_{n \in Z_0 \cap Z_1} (W_n \cap (Y_n \setminus \{x_n\}) \cup \{x_\ast\}$. 

We show $W \in \tau^0$. 
(i) For each $n \in \mathbb{N}$, 
\[
	W \cap (Y_n \setminus \{ x_n \}) = \begin{cases} W_n \cap (Y_n \setminus \{ x_n \}) & \text{if}\ n \in Z_0 \cap Z_1; \\ \varnothing & \text{otherwise}, \end{cases}
\]
and this set is in $\tau_n^0$. 
(ii) If $n \in Z_0 \cap Z_1$, then $W \restriction Y_n = W_n \cap Y_n \in \tau_n^0$. 
Thus $Z_0 \cap Z_1 \subseteq \{ n \in \mathbb{N} \mid W \restriction Y_n \in \tau_n^0 \} \in \mathcal{U}$ because $\mathcal{U}$ is a filter. 
Therefore $W \in \tau^0$.

Since $x_\ast \in d_{\tau^0} (v(\psi) \cap U)$, there exists $y \neq x_\ast$ such that $y \in v(\psi) \cap U \cap W$. 
Since $y \in W$, there exists $m \in Z_0 \cap Z_1$ such that $y \in v(\psi) \cap U_m \cap W_m \cap (Y_m \setminus \{ x_m \})$. 
Applying Lemma \ref{TruthL}, $y \in v_m(\psi) \cap U_m \cap W_m \cap (Y_m \setminus \{ x_m \})$. 
Then $y \neq x_m$ and $y \in v_m(\psi) \cap U_m \cap W_m$. 
This contradicts $v_m(\psi) \cap U_m \cap W_m \subseteq \{ x_m \}$. 
Our proof of Claim \ref{Cl6} is completed.
\end{proof}

We continue the proof of $x_\ast \in d_{\tau^0}(v(\chi) \cap U)$. 
From Claim \ref{Cl6}, there exists $n \in Z_0 \cap Z_1$ such that $x_n \in d_{\tau_n^0}(v_n (\psi) \cap U_n)$. 
Since $n \in Z_0$, we have $x_n \in v_n (\psi \rhd \chi)$. 
Therefore $x_n \in d_{\tau_n^0} (v_n(\chi) \cap U_n)$. 
Moreover, since $n \in Z_1$, we have $V \restriction Y_n \in \tau_n^0$. 
This set is a $\tau_n^0$-neighborhood of $x_n$, and thus there exists $y \neq x_n$ such that $y \in v_n (\chi) \cap U_n \cap (V \restriction Y_n)$. 
Since $y \neq x_n$, we obtain $y \in v(\chi) \cap U_n \cap V \cap (Y_n \setminus \{x_n\})$ by Lemma \ref{TruthL}. 
In particular, $y \neq x_\ast$ and $y \in v(\chi) \cap U \cap V$. 
This implies $x_\ast \in d_{\tau^0}(v(\chi) \cap U)$. 
We conclude $x_\ast \in v(\psi \rhd \chi)$. 
\end{proof}

The following lemma is an adaptation of Shehtman's result on the preservation of validity in ultrabouquets to our framework (See Shehtman \cite[Lemma 5.6]{She98}).

\begin{lem}\label{preservation1}
If an $\mathcal{L}(\Box, \rhd)$-formula $\varphi$ is valid in all $\langle X_n, \tau^0_n, \tau^1_n \rangle$, then for all valuations $v'$ on $\mathfrak{X}$ and all $x \in Y$, $x \in v'(\varphi)$. 
\end{lem}
\begin{proof}
We prove the contrapositive. 
Suppose that there exist a valuation $v'$ on $\mathfrak{X}$ and $x \in Y$ such that $x \notin v'(\varphi)$. 
For each $n \in \mathbb{N}$, we define a valuation $v_n'$ on $\langle X_n, \tau^0_n, \tau^1_n \rangle$ as follows: 
\begin{itemize}
\item For $x \in X_n \setminus \{x_n\}$, $x \in v_n' (p) : \iff x \in v' (p)$; 
\item $x_n \in v_n'(p) : \iff x_\ast \in v' (p)$.
\end{itemize}
Then the valuation on $\mathfrak{X}$ defined from $\{v_n'\}_{n \in \mathbb{N}}$ in Definition \ref{bouquetval} coincides with $v'$ because $\varnothing \notin \mathcal{U}$ and $\mathbb{N} \in \mathcal{U}$. 
We distinguish the following two cases. 

If $x \in Y_n \setminus \{x_n\}$, then by Lemma \ref{TruthL}, we obtain $x \notin v_n'(\varphi)$. 

If $x = x_\ast$, then by Lemma \ref{TruthL2}, $\{ n \in \mathbb{N} \mid x_n \in v_n'(\varphi) \} \notin \mathcal{U}$. 
Since $\mathbb{N} \in \mathcal{U}$, for some $n \in \mathbb{N}$, $x_n \notin v_n'(\varphi)$. 

Thus in either case, $\varphi$ is not valid in $\langle X_n, \tau^0_n, \tau^1_n \rangle$ for some $n \in \mathbb{N}$. 
\end{proof}

From the viewpoint of Lemma \ref{deadend}, the set $Y$ in the statement of Lemma \ref{preservation1} does not seem to be replaceable by $X$ in general. 
However, we prove that this is actually the case.  
First, we prove that the validity of the axiom $\Box(\Box p \to p) \to \Box p$ of $\GL$ is preserved. 

\begin{lem}\label{bitop_sc}
The topological space $\langle X, \tau^0 \rangle$ is scattered. 
That is, the ultrabouquet $\mathfrak{X}$ is a $\CL$-space. 
\end{lem}
\begin{proof}
Since each space $\langle X_n, \tau^0_n, \tau^1_n \rangle$ is scattered, $\varphi : \equiv \Box(\Box p \to p) \to \Box p$ is valid in $\langle X_n, \tau^0_n, \tau^1_n \rangle$ by Fact \ref{SimEsa}. 
Let $v'$ be any valuation on $\mathfrak{X}$. 
By Lemma \ref{preservation1}, for all $y \in Y$, $y \in v'(\varphi)$. 
Moreover, by Lemma \ref{deadend}, for all $x \in X \setminus Y$, $x \in cd_{\tau^0}(v'(p))$, that is, $x \in v'(\Box p)$. 
Hence $x \in v'(\varphi)$. 
Thus $\varphi$ is valid in $\mathfrak{X}$, and hence $\GL \subseteq \Log(\mathfrak{X})$. 
We conclude that $\langle X, \tau^0 \rangle$ is scattered. 
\end{proof}

The following lemma is a version of a part of Makinson's theorem (See Makinson \cite{Mak71}). 
Our proof is a modification of that in Hughes and Cresswell \cite[Lemma 3.2]{HC96}). 

\begin{lem}\label{preservation2}
Let $L$ be any consistent normal extension of $\CL$ and $\varphi$ be any $\mathcal{L}(\Box, \rhd)$-formula. 
If $\varphi \in L$, then $\Box \bot \to \varphi \in \CL$. 
\end{lem}
\begin{proof}
Let $L$ be a normal extension of $\CL$ and suppose that there exists an $\mathcal{L}(\Box, \rhd)$-formula $\varphi$ such that $\varphi \in L$ but $\Box \bot \to \varphi \notin \CL$. 
We would like to show that $L$ is inconsistent. 
From axioms $\J{1}$ and $\J{4}$, we have that for any $\mathcal{L}(\Box, \rhd)$-formula $\psi$, $\Box \psi$ is equivalent to $(\neg \psi) \rhd \bot$ in $\CL$. 
So we may assume that neither $\Box$ nor $\Diamond$ occurs in $\varphi$. 
Also we assume that $\varphi$ is in a conjunctive normal form $\varphi_0 \land \varphi_1 \land \cdots \land \varphi_k$ where each $\varphi_i$ is a disjunction of formulas, and each disjunct of $\varphi_i$ is either a formula without $\rhd$, or a formula of the form $\psi \rhd \chi$, or a formula of the form $\neg (\psi \rhd \chi)$. 

By the choice of $\varphi$, for some $i \leq k$, $\varphi_i \in L$ and $\Box \bot \to \varphi_i \notin \CL$. 
From $\J{1}$, we have that $\Box \bot \to \psi \rhd \chi \in \CL$. 
Then $\varphi_i$ does not contain a formula of the form $\psi \rhd \chi$ as a disjunct because $\Box \bot \to \varphi_i \notin \CL$. 
Thus, we may assume that $\varphi_i$ is of the form
\[
	\gamma \lor \bigvee_{j = 0}^{m} \neg (\psi_j \rhd \chi_j)
\]
where $\gamma$ is a classical propositional formula. 
Since $\Box \bot \to \varphi_i \notin \CL$, $\gamma$ is not a tautology of the classical propositional logic.
Then, there exists a substitution instance $\gamma'$ of $\gamma$ such that $\neg \gamma'$ is a tautology (cf.~\cite[p.~47]{HC96}). 
So $\neg \gamma' \in L$. 

Suppose $m = 0$. 
Then $L$ contains both $\gamma'$ and $\neg \gamma'$, and hence is inconsistent.

Suppose $m > 0$. 
Since each $\neg(\psi_j \rhd \chi_j)$ implies $\Diamond \top$ in $\CL$, $L$ contains $\gamma \lor \Diamond \top$. 
Then $\gamma' \lor \Diamond \top \in L$, and thus $\Diamond \top \in L$. 
Since $L$ is normal, $\Box \Diamond \top \in L$. 
Therefore $\Box \bot \in L$ because $L$ is an extension of $\CL$. 
We conclude that $L$ is inconsistent. 
\end{proof}

\begin{thm}\label{preservation3}
If an $\mathcal{L}(\Box, \rhd)$-formula $\varphi$ is valid in all $\langle X_n, \tau^0_n, \tau^1_n \rangle$, then $\varphi$ is also valid in $\mathfrak{X}$. 
\end{thm}
\begin{proof}
Since $\langle X_0, \tau^0_0 \rangle$ is scattered, $\Log(X_0, \tau^0_0, \tau^1_0)$ is a consistent normal extension of $\CL$ by Corollary \ref{Sound_CL}. 
Since $\varphi \in \Log(X_0, \tau^0_0, \tau^1_0)$, we obtain $\Box \bot \to \varphi \in \CL$ by Lemma \ref{preservation2}. 

Let $v'$ be any valuation on $\mathfrak{X}$, then for all $y \in Y$, $y \in v'(\varphi)$ by Lemma \ref{preservation1}. 
Also, for all $x \in X \setminus Y$, $x \in v'(\Box \bot)$ by Lemma \ref{deadend}. 
Since $\mathfrak{X}$ is a $\CL$-space by Lemma \ref{bitop_sc}, it follows from $\Box \bot \to \varphi \in \CL$ that $x \in v'(\varphi)$. 
Therefore $\varphi$ is valid in $\mathfrak{X}$. 
\end{proof}

\subsection{Proofs of the theorems}\label{proofs_CL}

We are ready to prove the topological compactness theorem. 

\begin{thm}[Topological compactness theorem]\label{CompCL}
Let $L$ be a consistent normal extension of $\CL$, $\Gamma$ be a set of $\mathcal{L}(\Box, \rhd)$-formulas and $\varphi$ be an $\mathcal{L}(\Box, \rhd)$-formula. 
If $\Gamma \models_L^{T} \varphi$, then $\Gamma_0 \models_L^{T} \varphi$ for some finite subset $\Gamma_0$ of $\Gamma$.
\end{thm}
\begin{proof}
Suppose that for all finite subsets $\Gamma_0$ of $\Gamma$, $\Gamma_0 \not\models_L^T \varphi$. 
Let $\{ \psi_n \}_{n \in \mathbb{N}}$ be an enumeration of elements of $\Gamma$, and let $\chi_n := \bigwedge_{i=0}^n \psi_n$. 
Then, for each $n \in \mathbb{N}$, $\{\chi_n\} \not \models_L^T \varphi$. 
Hence there exist an $L$-space $\langle X_n, \tau_n^0, \tau_n^1 \rangle$, a valuation $v_n$ on the space and $x_n \in X_n$ such that $x_n \in v_n (\chi_n)$ and $x_n \not\in v_n (\varphi)$. 
By Corollary \ref{Sound_CL}, $\langle X_n, \tau_n^0 \rangle$ is scattered. 
Also we may assume that the family $\{X_n\}_{n \in \mathbb{N}}$ is pairwise disjoint. 
Let $\mathfrak{X}$ be an ultrabouquet based on the families $\{ \langle X_n, \tau_n^0, \tau_n^1 \rangle \}_{n\in\mathbb{N}}$ and $\{x_n\}_{n \in \mathbb{N}}$. 
Since every $\varphi \in L$ is valid in all $\langle X_n, \tau^0_n, \tau^1_n \rangle$, by Lemma \ref{preservation3}, $\varphi$ is also valid in $\mathfrak{X}$. 
Therefore $\mathfrak{X}$ is also an $L$-space. 

Let $v$ be the valuation on $\mathfrak{X}$ defined from $\{v_n\}_{n \in \mathbb{N}}$ in Definition \ref{bouquetval}. 
We claim that for every $\psi_i \in \Gamma$, $x_\ast \in v(\psi_i)$. 
Indeed, for any $n \geq i$, $x_n \in v_n(\psi_i)$. 
Then the set $\{ n \in \mathbb{N} \mid x_n \in v_n(\psi_i)\}$ is cofinite, and hence in $\mathcal{U}$ because $\mathcal{U}$ is a non-principal ultrafilter. 
By Lemma \ref{TruthL2}, $x_\ast \in v(\psi_i)$. 

On the other hand, $\{ n \in \mathbb{N} \mid x_n \in v_n(\varphi) \} = \varnothing \not\in \mathcal{U}$. Again by Lemma \ref{TruthL2}, $x_\ast \not\in v(\varphi)$. Thus we conclude $\Gamma \not\models_L^T \varphi$.
\end{proof}

\begin{thm}\label{equivTC}
For any normal extension $L$ of $\CL$, $L$ is topologically complete if and only if $L$ is topologically strongly complete. 
\end{thm}

\begin{proof}
It suffices to prove the implication $(\Rightarrow)$. 
Suppose $\Gamma \models_L^{T} \varphi$. 
By the topological compactness theorem, $\Gamma_0 \models_L^T \varphi$ for some finite subset $\Gamma_0$ of $\Gamma$, and we have $\varnothing \models_L^{T} \bigwedge \Gamma_0 \to \varphi$. 
By the topological completeness of $L$, $\varnothing \vdash_L \bigwedge \Gamma_0 \to \varphi$. 
Thus $\Gamma \vdash_L \varphi$.
\end{proof}

From Theorems \ref{wTCCL} and \ref{equivTC}, we obtain the following topological strong completeness theorem. 

\begin{thm}[Topological strong completeness theorem of some extensions of $\CL$]\label{sTCCL}
The logics $\CL$, $\CLM$, $\IL$, $\ILM$, $\ILP$ and $\ILW$ are topologically strongly complete. 
\qed
\end{thm}

\section{Topological investigations of $\IL$}\label{IL}

In this section, we investigate topological aspects of $\IL$. 
First, we investigate necessary and sufficient conditions for a $\CL$-space to be an $\IL$-space. 
Secondly, we explore Alexandroff $\IL$-spaces. 

\begin{thm}\label{IL_Char}
Let $\langle X, \tau^0, \tau^1 \rangle$ be a $\CL$-space. 
Then the following are equivalent:
\begin{enumerate}
	\item $\langle X, \tau^0, \tau^1 \rangle$ is an $\IL$-space. 
	\item For any $U \in \tau^1$ and $Y \subseteq X$, $d_{\tau^0}(d_{\tau^0}(Y) \cap U) \subseteq d_{\tau^0}(Y \cap U)$. 
	\item For any $U \in \tau^1$, $d_{\tau^0}(d_{\tau^0}(\overline{U}) \cap U) = \varnothing$. 
	\item For any $U \in \tau^1$, there exists $V \in \tau^0$ such that $V \subseteq U$ and $d_{\tau^0}(U \setminus V) = \varnothing$. 
\end{enumerate}
\end{thm}

\begin{proof}
$(1 \Leftrightarrow 2)$: 
Notice that a $\CL$-space $\langle X, \tau^0, \tau^1 \rangle$ is an $\IL$-space if and only if $\Diamond p \rhd p$ is valid in $\langle X, \tau^0, \tau^1 \rangle$. 
The latter condition is equivalent to the condition that for all $Y \subseteq X$, $e_{\tau^0, \tau^1}(d_{\tau^0}(Y), Y) = X$. 
Then it follows from Definition \ref{ope} that this is equivalent to Clause 2. 

$(2 \Rightarrow 3)$: 
Let $U \in \tau^1$. 
From Clause 2 for $Y = \overline{U}$, we have $d_{\tau^0}(d_{\tau^0}(\overline{U}) \cap U) \subseteq d_{\tau^0}(\overline{U} \cap U) = d_{\tau^0}(\varnothing)$. 
Since $d_{\tau^0}(\varnothing) = \varnothing$ by Fact \ref{dset}.1, we obtain $d_{\tau^0}(d_{\tau^0}(\overline{U}) \cap U) = \varnothing$. 

$(3 \Rightarrow 2)$: 
Let $U \in \tau^1$ and $Y \subseteq X$. 
Since $Y \setminus U \subseteq \overline{U}$, by Fact \ref{dset}.2, $d_{\tau^0}(Y \setminus U) \cap U \subseteq d_{\tau^0}(\overline{U}) \cap U$. 
Then $d_{\tau^0}(d_{\tau^0}(Y \setminus U) \cap U) \subseteq d_{\tau^0}(d_{\tau^0}(\overline{U}) \cap U) = \varnothing$. 
We get $d_{\tau^0}(d_{\tau^0}(Y \setminus U) \cap U) = \varnothing$.

Since $Y = (Y \cap U) \cup (Y \setminus U)$, by Fact \ref{dset},
\begin{align*}
	d_{\tau^0}(d_{\tau^0}(Y) \cap U) & = d_{\tau^0}(d_{\tau^0}(Y \cap U) \cap U) \cup d_{\tau^0}(d_{\tau^0}(Y \setminus U) \cap U), \\
	& = d_{\tau^0}(d_{\tau^0}(Y \cap U) \cap U), \\
	& \subseteq d_{\tau^0}(d_{\tau^0}(Y \cap U)), \\
	& \subseteq d_{\tau^0}(Y \cap U). 
\end{align*}

$(3 \Rightarrow 4)$: 
Let $U \in \tau^1$. 
Let $V$ denote the set $U \setminus d_{\tau^0}(\overline{U})$. 
Then $V \subseteq U$ and $d_{\tau^0}(U \setminus V) = d_{\tau^0}(d_{\tau^0}(\overline{U}) \cap U) = \varnothing$. 
So it suffices to show that $V$ is an element of $\tau^0$. 
Let $x$ be an arbitrary element of $V$. 
Since $x \notin d_{\tau^0}(d_{\tau^0}(\overline{U}) \cap U)$, there exists a $\tau^0$-neighborhood $W_0$ of $x$ such that $W_0 \cap d_{\tau^0}(\overline{U}) \cap U \subseteq \{x\}$. 
Since $x \notin d_{\tau^0}(\overline{U})$, $W_0 \cap d_{\tau^0}(\overline{U}) \cap U = \varnothing$. 
Furthermore, from $x \notin d_{\tau^0}(\overline{U})$, there exists a $\tau^0$-neighborhood $W_1$ of $x$ such that $W_1 \cap \overline{U} \subseteq \{x\}$. 
Since $x \notin \overline{U}$, we also have $W_1 \cap \overline{U} = \varnothing$. 
Equivalently, $W_1 \subseteq U$. 
Then, we have $W_0 \cap W_1 \in \tau^0$, $x \in W_0 \cap W_1$ and $W_0 \cap W_1 \subseteq V$. 
We have shown that an arbitrary element of $V$ has a $\tau^0$-neighborhood which is included in $V$. 
Therefore $V \in \tau^0$. 

$(4 \Rightarrow 3)$: 
Let $U \in \tau^1$, then for some $V \in \tau^0$, $V \subseteq U$ and $d_{\tau^0}(U \setminus V) = \varnothing$. 
Since $\overline{U} \subseteq \overline{V}$ and $V \in \tau^0$, by Fact \ref{dset}, $d_{\tau^0}(\overline{U}) \cap V \subseteq d_{\tau^0}(\overline{V}) \cap V = \varnothing$. 
Then $d_{\tau^0}(\overline{U}) \cap V = \varnothing$ and so $d_{\tau^0}(d_{\tau^0}(\overline{U}) \cap V) = \varnothing$. 

Since $U = V \cup (U \setminus V)$, we obtain
\begin{align*}
	d_{\tau^0}(d_{\tau^0}(\overline{U}) \cap U) & = d_{\tau^0}(d_{\tau^0}(\overline{U}) \cap V) \cup d_{\tau^0}(d_{\tau^0}(\overline{U}) \cap (U \setminus V)), \\
	& = d_{\tau^0}(d_{\tau^0}(\overline{U}) \cap (U \setminus V)), \\
	& \subseteq d_{\tau^0}(U \setminus V). 
\end{align*}
Therefore we conclude $d_{\tau^0}(d_{\tau^0}(\overline{U}) \cap U) = \varnothing$. 
\end{proof}

\begin{cor}\label{IL_subset}
For any $\CL$-space $\langle X, \tau^0, \tau^1 \rangle$, if $\tau^1 \subseteq \tau^0$, then $\langle X, \tau^0, \tau^1 \rangle$ is an $\IL$-space. 
\end{cor}
\begin{proof}
Let $U \in \tau^1$, then $U \in \tau^0$. 
By Fact \ref{dset}, $d_{\tau^0}(\overline{U}) \cap U = \varnothing$, and hence $d_{\tau^0}(d_{\tau^0}(\overline{U}) \cap U) = \varnothing$. 
By Theorem \ref{IL_Char}, $\langle X, \tau^0, \tau^1 \rangle$ is an $\IL$-space. 
\end{proof}

We have already stated that $\IL$ is complete with respect to Visser semantics (Fact \ref{ComplIL}). 
Actually, Visser proved the following stronger result saying that $\IL$ is sound and complete with respect to a smaller class of Visser frames than the class of all Visser frames validating $\IL$ (See also Fact \ref{VisSem}.3). 

\begin{fact}[Visser \cite{Vis88}]\label{ComplIL2}
For any $\mathcal{L}(\Box, \rhd)$-formula $\varphi$, the following are equivalent: 
\begin{enumerate}
	\item $\varnothing \vdash_\IL \varphi$. 
	\item $\varphi$ is valid in all Visser frames $\langle W, R, S \rangle$ with $R \subseteq S$. 
\end{enumerate}
\qed
\end{fact}

We explain how Fact \ref{ComplIL2} follows from Fact \ref{ComplIL} in our framework. 
For this purpose, we prepare the following lemmas. 

\begin{lem}\label{Alex3}
For any topological space $\langle X, \tau \rangle$, the following are equivalent: 
\begin{enumerate}
	\item $\langle X, \tau \rangle$ is Alexandroff. 
	\item For any family $\{Y_i\}_{i \in I}$ of subsets of $X$, $d_{\tau}(\bigcup_{i \in I} Y_i) \subseteq \bigcup_{i \in I} d_{\tau}(Y_i)$. 
\end{enumerate}
\end{lem}
\begin{proof}
$(1 \Rightarrow 2)$: 
Let $\{Y_i\}_{i \in I}$ be any family of subsets of $X$. 
Let $x \notin \bigcup_{i \in I} d_\tau(Y_i)$. 
Then, for all $i \in I$, there exists a $\tau$-neighborhood $U_i$ of $x$ such that $Y_i \cap U_i \subseteq \{x\}$. 
Let $U : = \bigcap_{i \in I} U_i$, then $U$ is also a $\tau$-neighborhood of $x$ because $\tau$ is Alexandroff. 
Suppose, towards a contradiction, $x \in d_{\tau}(\bigcup_{i \in I}Y_i)$. 
Then there exists $y \neq x$ such that $y \in \left(\bigcup_{i \in I} Y_i \right) \cap U$. 
For some $j \in I$, $y \in Y_j \cap U \subseteq Y_j \cap U_j$, and this is a contradiction. 
Therefore $x \notin d_{\tau}(\bigcup_{i \in I}Y_i)$. 

$(2 \Rightarrow 1)$: 
Let $\{U_i\}_{i \in I}$ be any family of sets of $\tau$. 
Then for each $i \in I$, $d_{\tau}(\overline{U_i}) \cap U_i = \varnothing$ by Fact \ref{dset}.4. 
\begin{align*}
	d_{\tau}(\overline{\bigcap_{i \in I} U_i}) \cap \bigcap_{i \in I} U_i & = d_{\tau}(\bigcup_{i \in I} \overline{U_i}) \cap \bigcap_{i \in I} U_i, \\
	& \subseteq \bigcup_{i \in I} d_{\tau}(\overline{U_i}) \cap \bigcap_{i \in I} U_i, \tag{by Clause 1}\\
	& \subseteq \bigcup_{i \in I} (d_{\tau}(\overline{U_i}) \cap U_i) = \varnothing.
\end{align*}
Therefore $\bigcap_{i \in I} U_i$ is a member of $\tau$. 
\end{proof}

Notice that the converse inclusion $\bigcup_{i \in I} d_{\tau}(Y_i) \subseteq d_{\tau}(\bigcup_{i \in I}Y_i)$ in Lemma \ref{Alex3}.2 is easily obtained from Fact \ref{dset}.2. 

\begin{lem}\label{Alex4}
Let $\langle X, \tau \rangle$ be a topological space and $V, U \subseteq X$. 
If $V \subseteq U$ and $d_{\tau}(U \setminus V) = \varnothing$, then $d_{\tau}(Y \cap U) = d_{\tau}(Y \cap V)$ for all subsets $Y$ of $X$. 
\end{lem}
\begin{proof}
Notice that $d_{\tau^0}(Y \cap (U \setminus V))$ is also empty because $Y \cap (U \setminus V) \subseteq U \setminus V$. 
Since $U = (U \setminus V) \cup V$, by Fact \ref{dset}.3, 
\[
	d_\tau(Y \cap U) = d_\tau(Y \cap (U \setminus V)) \cup d_\tau(Y \cap V) = d_{\tau}(Y \cap V). 
\]
\end{proof}

\begin{thm}\label{Alex5}
Let $\langle X, \tau^0, \tau^1 \rangle$ be a bitopological space with both $\tau^0$ and $\tau^1$ are Alexandroff. 
Then, the following are equivalent:
\begin{enumerate}
	\item $\langle X, \tau^0, \tau^1 \rangle$ is an $\IL$-space. 
	\item $\tau^0$ is scattered and there exists an Alexandroff topology $\tau^2$ on $X$ such that $\tau^0 \cap \tau^1 \subseteq \tau^2 \subseteq \tau^0$ and $\Log (X, \tau^0, \tau^1) = \Log(X, \tau^0, \tau^2)$. 
	\item There exists a Visser frame $\langle X, R, S \rangle$ such that $R \subseteq S$ and $\Log(X, \tau^0, \tau^1) = \Log(X, R, S)$. 
\end{enumerate}
\end{thm}

\begin{proof}
$(1 \Rightarrow 2)$: Since $\langle X, \tau^0, \tau^1 \rangle$ is a $\CL$-space, $\tau^0$ is scattered by Corollary \ref{Sound_CL}. 
Define 
\[
	\tau^2 : = \{V \in \tau^0 \mid \exists U \in \tau^1[V \subseteq U\ \&\ d_{\tau^0}(U \setminus V) = \varnothing] \}.
\]
Then, obviously $\tau^2 \subseteq \tau^0$. 
Let $V \in \tau^0 \cap \tau^1$. 
Since $V \subseteq V$ and $d_{\tau^0}(V \setminus V) = d_{\tau^0}(\varnothing) = \varnothing$ by Fact \ref{dset}.1, we have $V \in \tau^2$. 
Thus $\tau^0 \cap \tau^1 \subseteq \tau^2$. 

First, we prove that $\tau^2$ is a topology on $X$. 
\begin{itemize}
	\item Since $X$ and $\varnothing$ are in $\tau^0 \cap \tau^1$, they are also in $\tau^2$. 
	\item Let $V_0, V_1 \in \tau^2$. 
	Then there exist elements $U_0$ and $U_1$ of $\tau^1$ such that $V_i \subseteq U_i$ for $i \in \{ 0,1 \}$ and $d_{\tau^0}(U_0 \setminus V_0) = d_{\tau^0}(U_1 \setminus V_1) = \varnothing$. 
	We have $V_0 \cap V_1 \subseteq U_0 \cap U_1 \in \tau^1$ and 
\begin{align*}
	d_{\tau^0}((U_0 \cap U_1) \setminus (V_0 \cap V_1)) & = d_{\tau^0}(((U_0 \cap U_1) \setminus V_0) \cup ((U_0 \cap U_1) \setminus V_1)), \\
	& \subseteq d_{\tau^0}((U_0 \setminus V_0) \cup (U_1 \setminus V_1)), \tag{by Fact \ref{dset}.2} \\
	& = d_{\tau^0}(U_0 \setminus V_0) \cup d_{\tau^0}(U_1 \setminus V_1) = \varnothing. \tag{by Fact \ref{dset}.3}
\end{align*}
Hence $V_0 \cap V_1 \in \tau^2$. 

	\item Let $\{V_i\}_{i \in I}$ be any family of elements of $\tau^2$. 
	Then for each $i \in I$, there exists $U_i \in \tau^1$ such that $V_i \subseteq U_i$ and $d_{\tau^0}(U_i \setminus V_i) = \varnothing$. 
	We get $\bigcup_{i \in I}V_i \subseteq \bigcup_{i \in I} U_i \in \tau^1$ and 
\begin{align*}
	d_{\tau^0}((\bigcup_{i \in I} U_i) \setminus (\bigcup_{i \in I} V_i)) & \subseteq d_{\tau^0}(\bigcup_{i \in I}(U_i \setminus V_i)), \tag{by Fact \ref{dset}.2} \\
	& \subseteq \bigcup_{i \in I}d_{\tau^0}(U_i \setminus V_i) = \varnothing. \tag{by Lemma \ref{Alex3}}
\end{align*}
Therefore $\bigcup_{i \in I} V_i$ is an element of $\tau^2$. 
\end{itemize}

Secondly, we prove $\tau^2$ is Alexandroff. 
Let $\{V_i\}_{i \in I}$ be any family of elements of $\tau^2$. 
Then for each $i \in I$, there exists $U_i \in \tau^1$ such that $d_{\tau^0}(U_i \setminus V_i) = \varnothing$. 
Since $\tau^1$ is Alexandroff, $\bigcap_{i \in I} V_i \subseteq \bigcap_{i \in I} U_i \in \tau^1$. 
Since $\tau^0$ is also Alexandroff, 
\begin{align*}
	d_{\tau^0}((\bigcap_{i \in I} U_i) \setminus (\bigcap_{i \in I} V_i)) & \subseteq d_{\tau^0}(\bigcup_{i \in I}(U_i \setminus V_i)), \tag{by Fact \ref{dset}.2} \\
	& = \bigcup_{i \in I}d_{\tau^0}(U_i \setminus V_i) = \varnothing. \tag{by Lemma \ref{Alex3}}
\end{align*}
Therefore $\bigcap_{i \in I}V_i \in \tau^2$. 

Finally, we prove $\Log(X, \tau^0, \tau^1) = \Log(X, \tau^0, \tau^2)$. 
It suffices to prove that for all subsets $Y, Z$ of $X$, $e_{\tau^0, \tau^1}(Y, Z) = e_{\tau^0, \tau^2}(Y, Z)$. 

$(\subseteq)$: 
Let $x \in e_{\tau^0, \tau^1}(Y, Z)$, $V \in \tau^2$ and $x \in d_{\tau^0}(Y \cap V)$. 
We would like to show $x \in d_{\tau^0}(Z \cap V)$. 
Then, there exists $U \in \tau^1$ such that $V \subseteq U$ and $d_{\tau^0}(U \setminus V) = \varnothing$. 
By Lemma \ref{Alex4}, $d_{\tau^0}(Y \cap U) = d_{\tau^0}(Y \cap V)$ and so $x \in d_{\tau^0}(Y \cap U)$. 
Since $x \in e_{\tau^0, \tau^1}(Y, Z)$, $x \in d_{\tau^0}(Z \cap U)$. 
By Lemma \ref{Alex4} again, $d_{\tau^0}(Z \cap U) = d_{\tau^0}(Z \cap V)$ and thus $x \in d_{\tau^0}(Z \cap V)$. 

$(\supseteq)$: 
Let $x \in e_{\tau^0, \tau^2}(Y, Z)$, $U \in \tau^1$ and $x \in d_{\tau^0}(Y \cap U)$. 
We would like to show $x \in d_{\tau^0}(Z \cap U)$. 
Since $\langle X, \tau^0, \tau^1 \rangle$ is an $\IL$-space, by Theorem \ref{IL_Char}, there exists $V \in \tau^0$ such that $V \subseteq U$ and $d_{\tau^0}(U \setminus V) = \varnothing$. 
Then, $V \in \tau^2$. 
As above, by Lemma \ref{Alex4}, $x \in d_{\tau^0}(Y \cap U) = d_{\tau^0}(Y \cap V)$. 
Since $x \in e_{\tau^0, \tau^2}(Y, Z)$, $x \in d_{\tau^0}(Z \cap V)$. 
Also by Lemma \ref{Alex4} again, $x \in d_{\tau^0}(Z \cap V) = d_{\tau^0}(Z \cap U)$. 

$(2 \Rightarrow 3)$: 
Let $R$ and $S$ be binary relations on $X$ defined as follows: 
\begin{itemize}
	\item $x R y : \iff x \neq y\ \&\ \forall U \in \tau^0 (x \in U \Rightarrow y \in U)$; 
	\item $x S y : \iff \forall U \in \tau^2 (x \in U \Rightarrow y \in U)$. 
\end{itemize}
As proved in the proof of Theorem \ref{Alex2}, $\langle X, R, S \rangle$ is a Visser frame, $\tau^0 = \tau_R$ and $\tau^2 = \tau_S$. 
By Corollary \ref{VSTS}, $\Log(X, R, S) = \Log(X, \tau_R, \tau_S) = \Log(X, \tau^0, \tau^2) = \Log(X, \tau^0, \tau^1)$. 
Also $R \subseteq S$ follows from the definitions of $R$ and $S$ and $\tau^2 \subseteq \tau^0$. 

$(3 \Rightarrow 1)$: 
This is a direct consequence of Fact \ref{VisSem}.3. 
\end{proof}

\begin{cor}
For any Visser frame $\langle W, R, S \rangle$, the following are equivalent: 
\begin{enumerate}
	\item $\IL \subseteq \Log(W, R, S)$. 
	\item There exists a Visser frame $\langle W, R, S' \rangle$ such that $R \subseteq S'$ and $\Log(W, R, S) = \Log(W, R, S')$. 
\end{enumerate}
\end{cor}
\begin{proof}
$(1 \Rightarrow 2)$: 
By Fact \ref{KSTS}, both $\tau_R$ and $\tau_S$ are Alexandroff. 
By Corollary \ref{VSTS}, $\Log(W, R, S) = \Log(W, \tau_R, \tau_S)$, and hence $\langle W, \tau_R, \tau_S \rangle$ is an $\IL$-space. 
By Theorem \ref{Alex5}, there exists a Visser frame $\langle W, R', S' \rangle$ such that $R' \subseteq S'$ and $\Log(W, R', S') = \Log(W, \tau_R, \tau_S)$. 
Then $\Log(W, R, S) = \Log(W, R', S')$. 
Furthermore, since $R$ is irreflexive and transitive, it is easily shown that for any $x, y \in W$, 
\[
	x R y \iff x \neq y\ \&\ \forall U \in \tau_R (x \in U \Rightarrow y \in U). 
\]
Notice that the right-to-left direction of this equivalence is proved by letting $U = \{ x \} \cup R(x)$.
From our proof of Theorem \ref{Alex5}, $R' = R$.

$(2 \Rightarrow 1)$: Immediate from Fact \ref{VisSem}.3. 
\end{proof}

\section{Concluding remarks}

In this paper, we newly introduced a topological semantics of $\CL$ and its extensions, and proved the topological compactness theorem. 
As a consequence, we proved that the logics $\CL$, $\CLM$, $\IL$, $\ILM$, $\ILP$ and $\ILW$ are strongly complete with respect to our topological semantics. 
These results are just the starting point for research in this direction.  
Obviously, investigating the topological completeness of other logics which are not listed above is an important further task. 

As we have described in Section \ref{TopSem}, we introduced our new topological semantics with Visser semantics in mind. 
Actually, we proved that every Visser frame can be considered as a topological frame (Corollary \ref{VSTS}). 
Also, each Visser frame can be considered as a Veltman frame, but it is not known whether each Veltman frame can be considered as a topological frame.
In this regard, we propose the following problem.

\begin{prob}
Is there a normal extension $L$ of $\CL$ such that $L$ is complete with respect to Veltman semantics but not with respect to our topological semantics?
\end{prob}

While $\CL$ and some of its extensions are strongly complete with respect to our semantics, they are not with respect to Veltman and Visser semantics. 
This seems to be an evidence that our semantics can provide more models than these relational semantics. 
Then, we expect an affirmative answer to the following problem. 

\begin{prob}
Is there a normal extension $L$ of $\CL$ such that $L$ is complete with respect to our semantics but not with respect to Veltman semantics?
\end{prob}

Visser \cite{Vis88} proved that the logics $\ILP$ and $\ILW$ have finite model property with respect to Visser semantics. 
That is, each of these logics is determined by a class of corresponding finite Visser frames. 
Therefore, these logics also have finite model property with respect to our topological semantics. 
On the other hand, Visser also proved that $\IL$ and $\ILM$ do not have finite model property with respect to Visser semantics (See also Visser \cite{Vis98}). 
Regarding this point, we propose the following problem. 

\begin{prob}
Do the logics $\CL$, $\CLM$, $\IL$ and $\ILM$ have finite model property with respect to our topological semantics?
\end{prob}

In order to understand the properties of axioms of $\CL$ and $\IL$ in more detail, the authors recently introduced several sublogics of them, and studied their basic characters such as completeness with respect to relational semantics and interpolation property (\cite{IKO,KO}).
We ask the following question about these sublogics.

\begin{prob}
Can we develop a topological semantics for these sublogics of $\CL$ and $\IL$?
\end{prob}

\section*{Acknowledgement}

The authors would like to thank Yuya Okawa for the valuable discussion. 
The second author was supported by JSPS KAKENHI Grant Number JP19K14586. 

\bibliographystyle{plain}
\bibliography{ref}

\end{document}